\documentclass[reqno,11pt]{amsart}
\usepackage{amsmath,amssymb,amsfonts,amsthm, amscd,indentfirst}
\usepackage{amsmath,latexsym,amssymb,amsmath,
	amscd,amsthm,amsxtra}
\usepackage{color}
\usepackage{pdfpages}
\usepackage{graphics}
\usepackage{enumitem}

\usepackage[usenames,dvipsnames]{xcolor}

\usepackage{varioref}
\usepackage{hyperref}
\usepackage[noabbrev,capitalise,nameinlink]{cleveref}

\hypersetup{colorlinks={true},linkcolor={blue},citecolor=blue}
\usepackage[english]{babel}
\usepackage{csquotes}

\usepackage[backend=biber,
style=alphabetic,datamodel=ext-eprint]{biblatex}
\DeclareSourcemap{
	\maps[datatype=bibtex]{
		\map{
			\step[fieldsource=pmid, fieldtarget=pubmed]
		}
	}
}
\makeatletter
\DeclareFieldFormat{arxiv}{%
	arXiv\addcolon\space
	\ifhyperref
	{\href{http://arxiv.org/\abx@arxivpath/#1}{%
			\nolinkurl{#1}%
			\iffieldundef{arxivclass}
			{}
			{\addspace\texttt{\mkbibbrackets{\thefield{arxivclass}}}}}}
	{\nolinkurl{#1}
		\iffieldundef{arxivclass}
		{}
		{\addspace\texttt{\mkbibbrackets{\thefield{arxivclass}}}}}}
\makeatother
\DeclareFieldFormat{pmcid}{%
	PMCID\addcolon\space
	\ifhyperref
	{\href{http://www.ncbi.nlm.nih.gov/pmc/articles/#1}{\nolinkurl{#1}}}
	{\nolinkurl{#1}}}
\DeclareFieldFormat{mr}{%
	MR\addcolon\space
	\ifhyperref
	{\href{http://www.ams.org/mathscinet-getitem?mr=MR#1}{\nolinkurl{#1}}}
	{\nolinkurl{#1}}}
\DeclareFieldFormat{zbl}{%
	Zbl\addcolon\space
	\ifhyperref
	{\href{http://zbmath.org/?q=an:#1}{\nolinkurl{#1}}}
	{\nolinkurl{#1}}}
\DeclareFieldAlias{jstor}{eprint:jstor}
\DeclareFieldAlias{hdl}{eprint:hdl}
\DeclareFieldAlias{pubmed}{eprint:pubmed}
\DeclareFieldAlias{googlebooks}{eprint:googlebooks}

\renewbibmacro*{eprint}{%
	\printfield{arxiv}%
	\newunit\newblock
	\printfield{jstor}%
	\newunit\newblock
	\printfield{mr}%
	\newunit\newblock
	\printfield{zbl}%
	\newunit\newblock
	\printfield{hdl}%
	\newunit\newblock
	\printfield{pubmed}%
	\newunit\newblock
	\printfield{pmcid}%
	\newunit\newblock
	\printfield{googlebooks}%
	\newunit\newblock
	\iffieldundef{eprinttype}
	{\printfield{eprint}}
	{\printfield[eprint:\strfield{eprinttype}]{eprint}}}

\usepackage{graphicx}
\usepackage{caption}

\usepackage{epsfig,here}
\usepackage{subfigure,here}

\newtheorem{theorem}{Theorem}[section]
\newtheorem{lemma}[theorem]{Lemma}
\newtheorem{proposition}[theorem]{Proposition}
\newtheorem{definition}[theorem]{Definition}
\newtheorem{remark}[theorem]{Remark}
\def\rr{\mathbf{R}}
\def\ss{\mathbf{S}}

\def\bb{\mathbf{B}}

\def\O{\Omega}
\def\p{\partial}

\def\b{\beta}

\def\p{\partial}

\def\<{\langle}
\def\>{\rangle}
\def\div{{\rm div}}
\def\n{\nabla}

\def\De{\Delta}

\def\ep{\epsilon}
\def\om{\omega}
\def\Om{\Omega}

\def\De{\Delta}
\def\oN{\overline{N}}
\def\onu{\overline{\nu}}

\def\p{\partial}

\newcommand{\mbB}{\mathbf{B}}

\newcommand{\mbR}{\mathbf{R}}
\newcommand{\mbS}{\mathbf{S}}

\newcommand{\mfk}{\mathbf{k}}
\newcommand{\mfB}{\mathbf{B}}
\newcommand{\mfR}{\mathbf{R}}
\newcommand{\mfn}{\overline{N}}

\newcommand{\mfW}{\mathbf{W}}
\newcommand{\mfS}{\mathbf{S}}

\newcommand{\mcF}{\mathcal{F}}
\newcommand{\mcG}{\mathcal{G}}
\newcommand{\mcH}{\mathcal{H}}

\newcommand{\mcL}{\mathcal{L}}

\newcommand{\ra}{\rightarrow}

\newcommand{\rd}{{\rm d}}
\newcommand{\rdt}{{\rm dt}}

\numberwithin{equation}{section}

\begin{document}
	\title[Volume-constraint local energy-minimizing sets]{Uniqueness for volume-constraint local energy-minimizing sets in a half-space or a ball}
	\author{Chao Xia}
	\address{School of Mathematical Sciences\\
		Xiamen University\\
		361005, Xiamen, P.R. China}
	\email{chaoxia@xmu.edu.cn}
	\author{Xuwen Zhang}
	\address{School of Mathematical Sciences\\
		Xiamen University\\
		361005, Xiamen, P.R. China}
	\email{xuwenzhang@stu.xmu.edu.cn}
	\thanks{This work is  supported by NSFC (Grant No. 11871406, 12271449). 
	}
	
	\begin{abstract}
		In this paper, we prove a Poincar\'e-type inequality for any set of finite perimeter which is stable with respect to the free energy among volume-preserving perturbation, provided that the Hausdorff dimension of its singular set is  at most $n-3$.
		With this inequality, we classify all the volume-constraint local energy-minimizing sets in a unit ball, a half-space or a wedge-shaped domain. In particular, we prove that the relative boundary of any energy-minimizing set is smooth.
		
		\
		
		
	\end{abstract}
	\maketitle
	
	\medskip
	

	\section{Introduction}
	
	The study of equilibrium shapes of a liquid confined in a given container has a long history. Since the work of Gauss, this subject has been studied through the introduction of a free energy functional. Precisely, for a liquid occupies a region $E$ inside a given container $\Om$, 
	its \textit{free energy} is given by
	\begin{align*}
		\sigma\left(P(E;\Om)-\beta P(E;\p\Om)\right)+\int_Eg(x)dx.
	\end{align*}
	
Mathematically, we assume $\Om\subset\mfR^n$ is a fixed connected open set with boundary $\p\O$ and $E$ is a set of finite perimeter in $\Om$.
Here $\sigma\in\mfR_+$ denotes the \textit{surface tension} at the interface between this liquid and other medium filling $\Om$, $\beta\in\mfR$ is called \textit{relative adhesion coefficient} between the fluid and the container, which satisfies $|\b|< 1$ due to Young's law, $g$ is typically assumed to be the \textit{gravitational energy}, whose integral is called the \textit{potential energy}.
 The free energy functional is usually minimized under volume constraint, that is, the enclosed volume $|E|$ is a constant. The existence of global minimizers of the free energy functional under volume constraint is easy to be shown by the direct method in calculus of variations, see for example \cite[Theorem 19.5]{Mag12}. 	
	For our purpose, we assume throughout this paper that $\sigma=1, g=0$, that is, we consider the energy functional
	\begin{eqnarray}\label{Free-energyFunctional}
		\mathcal{F}_\beta(E;\Om)=P(E;\Om)-\beta P(E;\p\Om), \quad |\beta|<1.
	\end{eqnarray}
	In the case that $\O=\mathbf{R}^n_+$, a half-space, the global minimizers of $\mathcal{F}_\beta$ under volume constraint has been classified by De Giorgi, see for example \cite[Theorem 19.21]{Mag12}. In the case that $\O=\mathbf{B}^n$, a unit ball, and $\beta=0$, the global minimizers under volume constraint has been classified long time ago by Burago-Maz'ya \cite{BM67} and Bokowsky-Sperner \cite{BS79}.

	Provided $\overline{\p E\cap \O}$ is sufficiently smooth, the boundary of stationary points of $E$ for the corresponding variational problems are capillary hypersurfaces $\overline{\p E\cap \O}$, namely, constant mean curvature hypersurfaces intersecting $\p \O$ at constant contact angle $\theta=\arccos \beta$. For the reader who are interested in the physical consideration of capillary surfaces, we refer to Finn's celebrated monograph \cite{Finn86} for a detailed account.
	
	When $\beta=0$, $\mathcal{F}_\beta(E;\Om)$ reduces to the perimeter functional $P(E;\Om)$ of $E$ in $\O$. The structure and regularity of local minimizers of $P(E;\Om)$ under volume constraint has been studied by Gonzalez-Massari-Tamanini \cite{GMT83} and Gr\"uter \cite{GJ86,Gruter87}. It was shown that for any local minimizer $E$, $\overline{\p E\cap \O}$ is smooth in $\O$ away from a singular set of Hausdorff dimension at most $n-8$.
	Moreover, Sternberg-Zumbrun \cite{SZ98} has derived a Poincar\'e-type inequality for any local minimizer $E$, provided the singular set in $\overline{\p E\cap\Om}$ is of Hausdorff dimension at most $n-3$.
	By using this Poincar\'e-type inequality, they proved the connectness of local minimizers in convex domains, smoothness of local minimizers in $\mathbf{R}^n_+$ \cite{SZ98}, and  smoothness of local minimizers in $\mbB^n$ under the additional condition $|E|<\frac{1}{n-1}\mathcal{H}^{n-1}(\bar E\cap \ss^{n-1})$ (Such condition has been recently verified by Barbosa \cite{Barbosa18}).
 Sternberg-Zumbrun \cite{SZ98} have conjectured all the local minimizers in a convex domain are smooth.
	On the other hand, they constructed a local minimizer with singularity in a non-convex domain \cite{SZ18}.
	Recently, Wang-Xia \cite{WX19} classified all local minimizers in $\mbB^n$ to be either totally geodesic balls or spherical caps intersecting $\ss^{n-1}$ orthogonally.
	In particular, they proved the smoothness of local minimizers in $\mbB^n$. The classification for $n=3$ has been proved by Nunes \cite{Nunes17}. Note that for $n=3$, the local minimizers are a priori known to be smooth by virtue of \cite{GMT83,Gruter87}.
		We remark that, in the smooth setting, that is, provided $\overline{\p E\cap \O}$ is $C^2$, the Poincar\'e-type inequality is just nonnegativity for the second variational formula for $P(E;\Om)=\mathcal{H}^{n-1}(\p E\cap \O)$ under volume constraint and the stability problem has been first investigated by Ros-Vergasta \cite{RV95}.
	
	In this paper, we study the general case $|\beta|<1$. In the smooth setting, $$\mathcal{F}_\beta(E;\Om)= \mathcal{H}^{n-1}(\p E\cap \O)- \beta\mathcal{H}^{n-1}(\p E\cap \p \O).$$  As we have already mentioned, the relative boundary  $\overline{\p E\cap \O}$ of a stationary point $E$ is a capillary hypersurface. The second variational formula for $\mathcal{F}_\beta$ under volume constraint has been derived by Ros-Souam \cite{RS97}. Wang-Xia \cite{WX19} and Souam \cite{AS16, Souam21} classified all smooth local minimizers in $\mbB^n$ and $\mathbf{R}^n_+$ respectively.
The key ingredient in Wang-Xia and Souam's proof is Minkowski-type formula which gives rise to suitable test functions that are used in the nonnegativity for second variational formula.
 
	As in the case $\beta=0$, the local minimizers of $\mathcal{F}_\beta$ under volume constraint in the case $\beta\neq 0$ are not known a priori to be smooth. Recently, it has been shown by De Philippis-Maggi \cite{DePM17}  that for any local minimizer $E$, $\overline{\p E\cap \O}$ is smooth in $\O$ away from a closed singular set of Hausdorff dimension at most $n-3$. 

	\begin{definition}\label{localmmmzer}
		\normalfont
		A set of finite perimeter $E\subset\Om$ is a \textit{local minimizer for the free energy functional \eqref{Free-energyFunctional} under volume constraint} if 
		\begin{align}
		\mcF_\beta(E;\Om)\leq\mcF_\beta(F;\Om),
		\end{align}
		among all sets of finite perimeter $F\subset\Om$ satisfying $\vert F\vert=\vert E\vert$ and $\vert F\De E\vert<\delta$ for some $\delta>0$.
	\end{definition}

 The main result in this paper is the classification of local minimizers of the free energy functional under volume constraint, when the container is a half-space or a ball.
	
	
	\begin{theorem}\label{umbillic}
		Let $\O$ be  $\rr^{n}_+$ or $\bb^{n}$.  Let $E\subset\O$ be a local minimizer for the free energy functional \eqref{Free-energyFunctional} under volume constraint among sets of finite perimeter.
		Then $M=\overline{\p E\cap\O}$ is (up to a modification of sets of measure zero for $E$) either part of a totally geodesic hyperplane or part of a sphere, which intersects with $\p\O$ at the contact angle $\theta=\arccos\beta$. In particular, $M$ is smooth.
	\end{theorem}
 
Our main strategy to prove \cref{umbillic} is as follows. 
First, following Sternberg-Zumbrun \cite{SZ98}, we prove a Poincar\'e-type inequality for local minimizers of $\mathcal{F}_\beta$ under volume constraint, see  \cref{Prop-StableSet}. In order to establish such inequality, we crucially make use of De Philippis-Maggi's \cite{DePM15, DePM17}  Hausdorff estimate for singular set and local Euclidean volume growth property for local minimizers to construct useful cut-off functions, see \cref{Lem-Cut-off}. 
We remark that the Poincar\'e-type inequality in  \cref{Prop-StableSet} holds provided some technical integrability condition \eqref{condi-integrability} on test functions. 
Second, we extend the Minkowski-type formula of Wang-Xia \cite{WX19} and Souam \cite{AS16, Souam21} to the  singular setting, see \cref{PropMinkowski-Poly} and \cref{Prop-Minko-Ball}. 
An important observation is that the test function arising from the Minkowski-type formula satisfies the integrability condition, which enables us to utilize the Poincar\'e-type inequality. Third, the same procedure of Wang-Xia \cite{WX19} and Souam \cite{AS16, Souam21} leads to the conclusion that ${\rm reg}M$ is spherical.

We remark that our proof for the half-space case also works for the wedge case. In fact, we shall handle the wedge case directly in \cref{Sec-5}. In the smooth setting, the corresponding stability problem has been investigated by Li-Xiong \cite{LX17} and Souam \cite{Souam21}.

The paper is organized as follows.
In \cref{Sec2} we recall some background materials about sets of finite perimeter and review a few useful results on local minimizers for the free energy functional recently developed by De Philippis-Maggi \cite{DePM15, DePM17}.
In \cref{Sec3} we construct the crucial cut-off functions in \cref{Lem-Cut-off}, and prove tangential divergence theorem on singular hypersurfaces.
In \cref{Sec-4}, we prove that the stationary set of the free energy functional under volume-constraint admits a singular capillary CMC hypersurface (\cref{Prop-capillaryCMC}), and the stable set admits the Poincar\'e-type inequality (\cref{Prop-StableSet}).
In \cref{Sec-5}, we prove \cref{umbillic} in the half-space case, and also in a more general setting, the wedge case.
In \cref{Sec-6}, we prove \cref{umbillic} in the ball case.
	
\vspace{0.5cm}

{\noindent \bf Acknowledgments.} The first author is grateful to Professor Guofang Wang for useful discussion on this subject and his constant support. We would like to thank Professor Peter Sternberg for answering our questions regarding their paper \cite{SZ98}. We also would like to thank the anonymous referee for pointing out to us the boundary regularity results by De Philipis and Maggi \cite{DePM15,DePM17} for local minimizers of anisotropic free energy functional under volume constraint.
	
	\
	
	\section{Preliminaries}\label{Sec2}

 \subsection{Notation}\label{Notation}
In all follows, we denote by $\left<\cdot,\cdot\right>$, ${\rm div},\nabla$, the inner product, the divergence operator, the gradient operator in $\mbR^n$, respectively. We denote by $\mcH^k$ the $k$-dimensional Hausdorff measure in $\mbR^{n}$. 
 We denote by $\mbB^n$  the $n$-dimensional unit ball, by $\mbS^{n-1}$ the $(n-1)$-dimensional unit sphere in $\mbR^n$,  by $B_r(x)$ a $n$-dimensional open ball in $\mbR^n$ with radius $r$ and centered at $x$, 
by $\omega_{n}$ the volume of $n$-dimensional Euclidean unit ball.

For a set $E\subset\mbR^n$, we denote by $\vert E\vert$ its $n$-dimensional Lebesgue measure,
$\chi_E$ denotes the indicator function of $E$.
We adopt the following notations when considering the topology of $\mfR^n$:
we denote
by $\overline{E}$ the topological closure of a set $E$, by ${\rm int}(E)$ the topological interior of $E$, by $E^c$ the topological complement of $E$, by $\p E$ the topological boundary of $E$ and by $E\De F$ the difference of two sets $E, F$.
In terms of the subspace topology (relative topology), we use the following notations.
 Let $X$ be a topological space and $S$ be a subspace of $X$. We use ${\rm cl}_XS, {\rm int}_XS, \p_XS$ to denote the closure, the interior and the boundary, respectively, of $S$ in the topological space $X$.
	
For the constraint problem, the container $\Om\subset\mbR^{n}$ is assumed to be a connected  (possibly unbounded) open set with $C^{2,\alpha}$-boundary $\p\Om$. Let $E\subset\Om$ be a set with finite volume and perimeter,
let $M$ denote the closed set $\overline{\p E\cap\Om}$, 
let $B^+$ denote the set $\p E\setminus M$, which is open in the subspace topology;
let $\Gamma$ denote the set $M\cap\p\Omega$.
let $\nu, \oN$ denote the outwards pointing unit normal of $M, B^+$, respectively, when they exist; $\mu, \onu$ denote the outwards pointing unit conormal of $\Gamma$ in $M,B^+$, respectively
	(see also \cref{figure1}).
Let $h$ denote the second fundamental form of ${\rm reg}M$ in $\mbR^n$ with respect to $-\nu$ (that is, $h(X,Y)=\left<\nabla_X\nu,Y\right>$ for any $X,Y\in TM$) and $h^{\p\Om}$ denotes the second fundamental form of $\p\Om$ with respect to the inwards pointing unit normal $-\oN$, $\vert\vert h\vert\vert^2=\sum_{i=1}^{n-1}\kappa_i^2$, where $\{\kappa_i\}$ are the principal curvatures of $M$. When taking an orthonormal basis $\{\tau_i\}_{i=1}^{n-1}$ on $TM$, the mean curvature $H$ of $M$ with respect to $h$ is given by $H=\sum_{i=1}^{n-1}h(\tau_i,\tau_i)$.
 
	\begin{figure}[h]
	\centering
	\includegraphics[height=8cm,width=14cm]{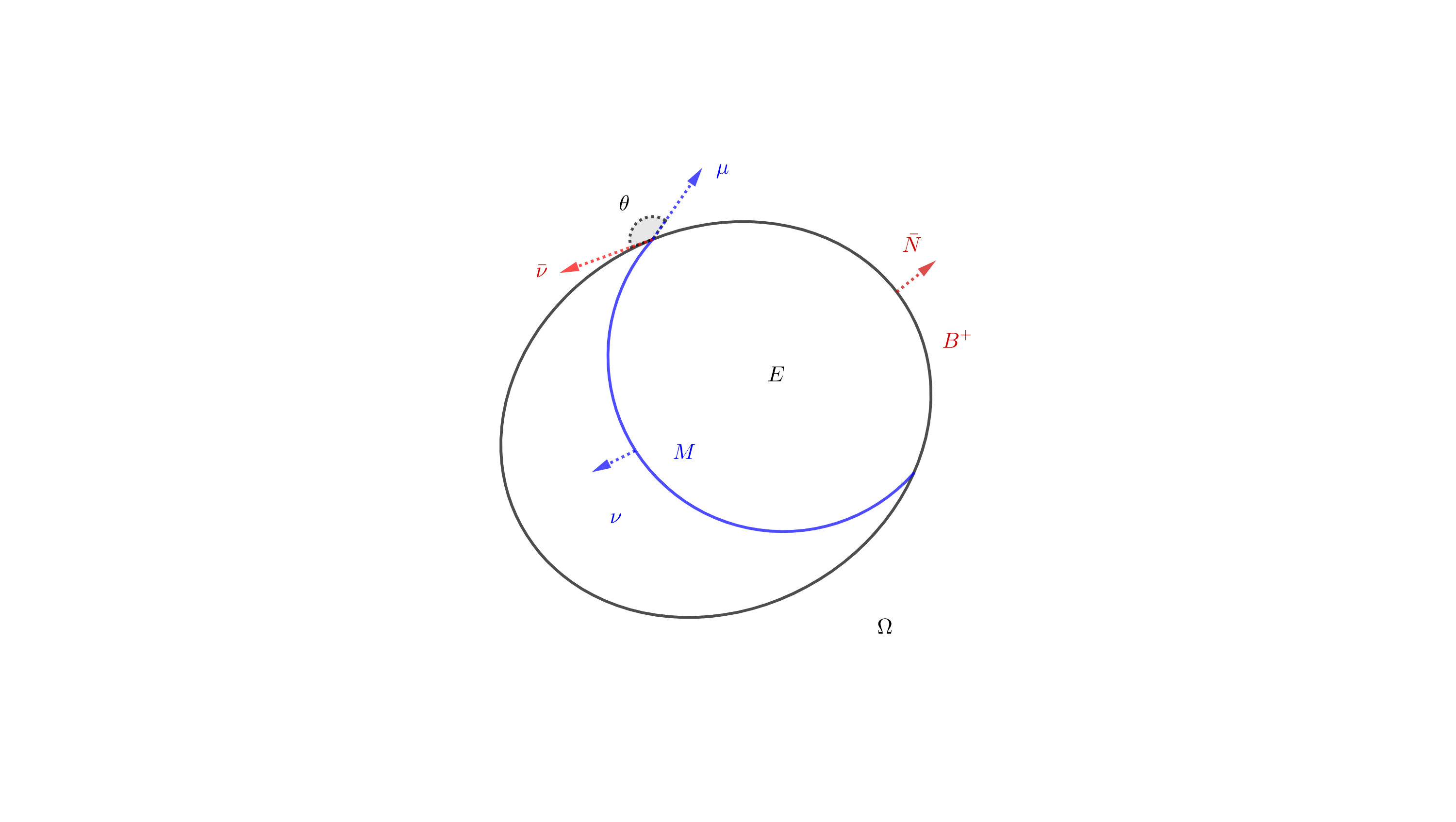}
	\caption{Notations}
	\label{figure1}
\end{figure}

	\subsection{Sets of finite perimeter}\label{Sec-sfp}
 In this subsection we collect some background materials for sets of finite perimeter, we refer to \cite[Chapter 17]{Mag12} for a detailed account.
	
	Let $E\subset \rr^n$ be a Lebesgue measurable set in $\mfR^n$, we say that \textit{$E$ is a set of finite perimeter in $\mfR^n$} if
 \begin{align*}
     \sup\left\{\int_{\mfR^n} {\rm div}X\rd\mcL^{n}: X\in C^1_c(\mfR^n;\mfR^n),  \vert X\vert\leq1\right\}<\infty.
 \end{align*}
 An equivalent characterization of sets of finite perimeter (see \cite[Proposition 12.1]{Mag12}) is that: there exists a $\mfR^n$-valued Radon measure $\mu_E$ on $\mfR^n$ such that for any $X\in C_c^1(\mfR^n;\mfR^n)$,
 \begin{align*}
     \int_E{\rm div}X
     =\int_{\mfR^n}\left<X,\rd\mu_E\right>.
 \end{align*}
$\mu_E$ is called the \textit{Gauss-Green measure of $E$}.
The \textit{relative perimeter of $E$ in $F\subset\mfR^n$}, and the \textit{perimeter of $E$}, are defined as
\begin{align*}
    P(E;F)=\vert\mu_E\vert(F),\quad P(E)=\vert\mu_E\vert(\mfR^n).
\end{align*}

Regarding the topological boundary of a set of finite perimeter $E$, one has (see \cite[Proposition 12.19]{Mag12})
\begin{align*}
    {\rm spt}\mu_E
    =\{x\in\mfR^n:0<\vert E\cap B_r(x)\vert<\om_nr^n,\quad\forall r>0\}\subset\p E.
\end{align*}

The \textit{reduced boundary} $\p^\ast E$ is the set of those $x\in{\rm spt}\mu_E$ such that the limit
\begin{align*}
    \lim_{r\ra0^+}\frac{\mu_E(B_r(x))}{\vert\mu_E\vert(B_r(x))}\text{ exists and belongs to }\mfS^{n-1}.
\end{align*}
A crucial fact (see \cite[(15.3)]{Mag12}) we shall use for our main result \cref{umbillic} is that, for a set of finite perimeter $E$, up to modification of sets of measure zero,
\begin{align*}
    \overline{\p^\ast E}=\p E.
\end{align*}
	\subsection{Regularity results for local minimizers} In this subsection, we summarize some known results for local minimizers of the free energy functional under volume constraint.
\begin{definition}
    \normalfont
    Let $\Om$ be an open, connected set in $\mfR^n$ with $C^{2,\alpha}$-boundary $\p\Om$, let $E\subset\Om$ be a set of finite perimeter and set $M=\overline{\p E\cap\Om}$.
    The \textit{regular part of $M$} is defined by
    \begin{align*}
        {\rm reg}M=\{x\in M:\text{ there exists an }r_x>0\text{ such that }{\rm reg}M\cap B_{r_x}(x)\\
        \text{is a }C^2\text{-manifold with boundary contained in }\p\Om\},
    \end{align*}
    while ${\rm sing}M=M\setminus{\rm reg}M$ is called the \textit{singular set of $M$}.
    In this way, ${\rm sing}M$ is relatively closed in $M$.
\end{definition}

The following Hausdorff dimensional estimate for singular sets of local minimizers has been proved by De Phillipis-Maggi \cite{DePM15,DePM17}.
\begin{theorem}[{\cite[Theorem 1.5, Lemma 2.5]{DePM17}}]\label{regularity-thm}
		Let $E\subset \O$ be a local minimizer of the free energy functional \eqref{Free-energyFunctional}  under volume constraint. Let $M=\overline{\p E\cap \O}$.  Then 
${\rm sing}M=\varnothing$ if $n=3$, while $\mcH^{n-3}({\rm sing}M)=0$ for any $n>3$. Moreover, the second fundamental form of ${\rm reg}M$ satisfies $\|h\|\in L^2(M)$.
	\end{theorem}
	\begin{remark}
 \normalfont
	    Note that De Phillipis-Maggi's result is stated for so-called almost-minimizers. Nevertheless, it is known that a local-minimizer under volume constraint is an almost-minimizer, see for example \cite[Example 21.3]{Mag12}.
	\end{remark}
 \begin{remark}\label{remark1}
		\normalfont
		Notice that by virtue of $\mcH^{n-3}({\rm sing}M)=0$, the integrals $\int_{M\cap\Om}\cdot \rd\mcH^{n-1}(x)$ and $\int_{{\rm reg}M\cap\Om}\cdot \rd\mcH^{n-1}$
		are exactly the same things.
		Also, since $B^+$ is $C^2$ in $\p\Om$ and $\Gamma= M\cap\p\Om$, we have: $\int_{B^+}\cdot \rd\mcH^{n-1}=\int_{{\rm reg}B^+}\cdot \rd\mcH^{n-1}$ and $\int_\Gamma\cdot \rd\mcH^{n-2}=\int_{\rm reg \Gamma}\cdot \rd\mcH^{n-2}$.
		Here ${\rm reg}B^+$ denotes the regular part of $B^+$ and ${\rm reg}\Gamma={\rm reg}M\cap \Gamma$.
	\end{remark}
 
 \begin{definition}[Euclidean volume growth]\label{Defn-LocalVolumeGrowth}
	    \normalfont
	    For a set of finite perimeter $E\subset\Om$,
	    let $M=\overline{\p E\cap\Om}$ and $\Gamma=M\cap\p\Omega$.
	      We say that \textit{$M$ satisfies the Euclidean volume growth condition} if for any $x\in M$,  there exists some universal constant (depending only on $\Om$ and $n$) $R_1>0$ and some universal constant $C_1>0$ such that for any $0<r<R_1$, there holds
	    \begin{align}\label{ineq-vg-M}
	        \mcH^{n-1}(M\cap B_r(x))\leq C_1r^{n-1}.
	    \end{align}
	     We say that \textit{$\Gamma$ satisfies the Euclidean volume growth condition} if for any $x\in \Gamma$,  there exists some universal constant $R_2>0$ and some universal constant $C_2>0$ such that for any $0<r<R_2$, there holds
	    \begin{align}\label{ineq-vg-Gamma}
	        \mcH^{n-2}(\Gamma\cap B_r(x))\leq C_2r^{n-2}.
	    \end{align}
	\end{definition}
 
We need the following result, due to De Phillipis-Maggi \cite{DePM15}, on the Euclidean volume growth for local minimizers.
	\begin{theorem}[{\cite[Lemma 2.8 and Lemma 2.10]{DePM15}}]\label{eucl-vol-growth}
		Let $E\subset \O$ be a local minimizer of the free energy functional \eqref{Free-energyFunctional}  under volume constraint.   Then $M=\overline{\p E\cap\Om}$ and $\Gamma=M\cap\p\Omega$ satisfy the Euclidean volume growth condition \eqref{ineq-vg-M} and \eqref{ineq-vg-Gamma} respectively.
\end{theorem}
\section{Cut-off functions}\label{Sec3}
 
In this section, we first construct cut-off functions near the singularities, under the assumption of Euclidean volume growth. The technique is standard and these cut-off functions are very useful for the study of surfaces with singularities, see e.g., \cite{SS81,Ilmanen96,Wickramasekera14,DePM17,Zhu18}.
\begin{lemma}[cut-off functions]\label{Lem-Cut-off}
Let $E\subset\Om$ be a set of finite perimeter. 
Assume that $M=\overline{\p E\cap\Om}$ and $\Gamma=M\cap\p\Omega$ satisfy the local Euclidean volume growth condition \eqref{ineq-vg-M} and \eqref{ineq-vg-Gamma} respectively. Assume in addition that $\mcH^{n-q-1}({\rm sing}M)=0$ for some $q>0$. 
 Then for any small $\epsilon>0$, there exist open sets $S_\epsilon'\subset S_\epsilon\subset\mbR^n$, with ${\rm sing}M\subset S_\epsilon'$ and $S_\epsilon\subset\{x:{\rm dist}(x,{\rm sing}M)<\epsilon\}$, and  a smooth cut-off function $\varphi_\epsilon\in C^\infty(\mbR^n)$ such that $0\leq\varphi_\epsilon\leq1$ with
		\begin{align}\label{varphi-construction}
			\varphi_\epsilon(x)=
			\begin{cases}
				0\quad &x\in S_\epsilon',\\
				1\quad &x\in\mbR^n\setminus S_\epsilon.
			\end{cases}
		\end{align}
  Moreover, $\varphi_\epsilon$ satisfies the following properties:
  \begin{align}\label{pw-conv}
\varphi_\ep(x)\to 1 \hbox{ pointwisely for }x\in {\rm reg}M,
\end{align}
	
\begin{align}
\int_M\vert\nabla^M\varphi_\epsilon(x)\vert^q \rd\mcH^{n-1}(x)\leq C\ep,\label{esti-cutoff-1}
\end{align}
	

\begin{align}\label{esti-cutoff-3}
\int_\Gamma\vert\nabla^M\varphi_\epsilon(x)\vert^{q-1} \rd\mcH^{n-1}(x)
\leq C\ep, \hbox{ for }q>1,
\end{align}
Here and in all follows, $C$ will be referred to as positive constants that are independent of $\ep$.
\end{lemma}
\begin{proof}
    We begin by noticing that ${\rm sing} M$ is compact since it is relatively closed and bounded.
			
	For any $\epsilon>0$, since $\mcH^{n-q-1}({\rm sing}M)=0$, we may cover the singular set ${\rm sing}M$ with finitely many balls $\mcG:=\{B_{r_i}(z_i)\}_{i=1}^{N_1}$ where $z_i\in M$, $\sum_{i=1}^{N_1} r_i^{n-q-1}<\epsilon$,
	and we may assume without loss of generality that $r_i<1$ for each $i$ and that $6r_i<\min\{R_1,R_2\}$, where $R_1,R_2$ are given in \cref{Defn-LocalVolumeGrowth}, within which the Euclidean volume growth conditions \eqref{ineq-vg-M} and \eqref{ineq-vg-Gamma} are valid for $z_i$.
 In particular, for those $B_{2r_i}(z_i)\cap\Gamma\neq\emptyset$, we may assume that $z_i\in\Gamma$, otherwise we may choose $\tilde z_i\in B_{2r_i}(z_i)\cap\Gamma$ and use $B_{3r_i}(\tilde z_i)$ to replace $B_{r_i}(z_i)$, since it follows directly that $B_{r_i}(z_i)\subset B_{3r_i}(\tilde z_i)$.
	Moreover, we set $C_M$ to be the smallest numbers among $C_1$ and $C_2$ given in \cref{Defn-LocalVolumeGrowth}, notice that $C_M$ is a universal constant that is independent of the choice of $z_i$ and $\ep$.

			For each $i$, let $\varphi_i\in C^\infty(\mbR^n)$ satisfy $0\leq\varphi_i\leq1$ with
			\begin{align*}
				\varphi_i(x)=\begin{cases}
					0\quad &\forall x\in B_{r_i}(z_i),\\
					1\quad &\forall x\in \mbR^n\setminus B_{2r_{i}}(z_i),
				\end{cases}
			\end{align*}
			and
   $$\vert \nabla\varphi_i(x)\vert\leq\frac{2}{r_i},
   \quad \hbox{ for all }x\in\mbR^n.$$
		Define $\tilde\varphi_\epsilon$ by
			\begin{align*}
	\tilde{\varphi}_\epsilon(x):=\min_i \varphi_i(x).
			\end{align*}
			It follows that $\tilde\varphi_\epsilon$ is piecewise-smooth with $0\leq\tilde{\varphi}_\epsilon\leq1$, and
			\begin{align}\label{varphitilde}
				\tilde{\varphi}_\epsilon(x)=
				\begin{cases}
					0 \quad \text{on}\quad  \bigcup_i B_{r_i}(z_i)\supseteq {\rm sing}M,\\
					1 \quad \text{on}\quad \mbR^n\setminus \bigcup_iB_{2r_i}(z_i).
				\end{cases}
			\end{align}
   It is clear that \eqref{varphi-construction} holds and hence \eqref{pw-conv} is true.
   
			By the Euclidean volume growth condition \eqref{ineq-vg-M} of $M$ and that $\sum_{i=1}^{N_1} r_i^{n-q-1}<\epsilon$, we have
			\begin{align}\label{ineq-cutoff-1}
				\int_M\vert\nabla^M\tilde{\varphi}_\epsilon(x)\vert^q\rd\mcH^{n-1}(x)
				&\leq\sum_{i}\int_{M\cap\left(B_{2r_i}(z_i)\setminus B_{r_i}(z_i)\right)}\vert\nabla\varphi_i(x)\vert^q\rd\mcH^{n-1}(x)\notag\\
				&\leq \sum_i \frac{2^q}{r_i^q}\mcH^{n-1}\left(M\cap B_{2r_i}(z_i)\right)\notag\\
				&\leq 2^{n+q-1}C_M\sum_i^{N_1} r_i^{n-q-1}\leq2^{n+q-1}C_M\epsilon.
			\end{align}
		For $q>1$,
		\begin{align}\label{ineq-cutoff-3}
			    \int_\Gamma\vert\nabla\tilde{\varphi}_\ep(x)\vert^{q-1}\rd\mcH^{n-2}(x)
			    \leq&\sum_i\frac{2^{q-1}}{r_i^{q-1}}\mcH^{n-2}(\Gamma\cap B_{2r_i}(z_i))\notag\\
			    \leq&2^{n+q-3}C_M\sum_i^{N_1}r_i^{n-q-1}
			    \leq2^{n+q-3}C_M\ep.
			\end{align}
  
  We mollify $\tilde{\varphi}_\epsilon$ to obtain a smooth function $\varphi_\epsilon$, which still satisfies estimates of the form \eqref{ineq-cutoff-1} 
  and \eqref{ineq-cutoff-3}. Since $\tilde{\varphi}_\epsilon$ satisfies \eqref{varphitilde}, we may let $S_\epsilon',S_\epsilon$ denote the sets such that
			\begin{align*}
				\varphi_\epsilon(x)=
				\begin{cases}
					0\quad &x\in S_\epsilon',\\
					1\quad &x\in\mbR^n\setminus S_\epsilon.
				\end{cases}
			\end{align*}
			We see that $\varphi_\epsilon$ is the desired smooth cut-off function, and this completes the proof.
\end{proof}

Using these cut-off functions,
we can prove the following tangential divergence theorem on hypersurfaces with singularities that are of low Hausdorff dimension and satisfying the Euclidean volume growth condition.
\begin{lemma}\label{Lem-divergenctheorem}
    Let $E\subset\Om$ be a set of finite perimeter. Assume that $M=\overline{\p E\cap\Om}$ and $\Gamma=M\cap\p\Omega$ satisfy the local Euclidean volume growth condition \eqref{ineq-vg-M} and \eqref{ineq-vg-Gamma} respectively. Assume in addition that  $\mcH^{n-2}({\rm sing}M)=0$
    and $H\in L^1(M).$  
    Then for any $X\in C^1(\mfR^n;\mfR^n)$, there holds
    \begin{align}\label{formu-div-M}
        \int_M{\rm div}_MX\rd\mcH^{n-1}=\int_M H\left<X,\nu\right>\rd\mcH^{n-1}+\int_\Gamma\left<X,\mu\right>\rd\mcH^{n-2},
    \end{align}
    and for any $X\in C^1(\mfR^n;\mfR^n)$ such that $X(x)\in T_x\p\Om$ on $\p\Om$,
    \begin{align}\label{formu-div-B+}
        \int_{B^+}{\rm div}_{\p\Om}X\rd\mcH^{n-1}=\int_\Gamma\left<X,\bar\mu\right>\rd\mcH^{n-2}.
    \end{align}
\end{lemma}
\begin{proof}
    Notice that this is the case when $q=1$ in \cref{Lem-Cut-off}.
    For any small $\ep>0$, we have $\varphi_\ep,S'_\ep$ and $S_\ep$ from \cref{Lem-Cut-off}.
    Let $X_\ep$ be a vector field given by
    \begin{align*}
        X_\ep:=\varphi_\ep X,
    \end{align*}
    We readily see that $X_\ep\in C^1(\mfR^n;\mfR^n)$ and
    \begin{align*}
        X_\ep
        =\begin{cases}
				0\quad&\text{on }S^{'}_\ep,\\
				\varphi_\ep X\quad&\text{on }S_\ep\setminus S^{'}_\ep,\\
				X\quad&\text{on }M\setminus S_\ep.
			\end{cases}
    \end{align*}
Integrating ${\rm div}_M(X_\ep)$ on $M\setminus S'_\ep$, we can apply the classical tangential divergence theorem to find
\begin{align*}
    \int_{M}{\rm div}_M X_\ep\rd\mcH^{n-1}
    =\int_{M}H\left<X_\ep,\nu\right>\rd\mcH^{n-1}+\int_{\Gamma}\left<X_\ep,\mu\right>\rd\mcH^{n-2}.
\end{align*}
A further computation then yields that
\begin{align*}
   \int_M \varphi_\epsilon \div_MX d\mcH^{n-1}+\int_{M}\left<\nabla^M\varphi_\ep,X\right>d\mcH^{n-1}
   =\int_M \varphi_\ep H\left<X,\nu\right>\rd\mcH^{n-1}+\int_\Gamma\varphi_\ep\left<X,\mu\right>\rd\mcH^{n-2}.
\end{align*}

Since $X\in C^1(\mfR^n;\mfR^n)$ and $M$ is bounded, we have that
${\rm div}_MX=({\rm div}X-\left<\nabla_\nu X,\nu\right>)$ and $\vert X\vert$ are bounded (the upper bounds are independent of $\ep$).
By virtue of \eqref{pw-conv} and \eqref{esti-cutoff-1} in \cref{Lem-Cut-off}, and the assumption that $H\in L^1(M)$, we may send $\ep\searrow0$ and use the dominated convergence theorem to conclude \eqref{formu-div-M}.

On the other hand, since $B^+\subset\p\Om$ is $C^2$ and the singularities of ${\rm cl}_{\p\Om}B^+$ are on $\Gamma$, thus we can follow the proof of \eqref{formu-div-M} to conclude \eqref{formu-div-B+}. This completes the proof. 
\end{proof}
    Next we establish a useful tool for the study of hypersurface with boundary in differential geometry, which is well-known and widely used in the smooth setting (see for example \cite{AS16,LX17,Souam21}).
	 Thanks to the cut-off functions, we can extend this classical result to the singular setting.
	 
	\begin{lemma}\label{PropAS16-2.4}
		Let $E\subset\Om$ be a set of finite perimeter. 
		Assume that $M$ and $\Gamma$ satisfy the local Euclidean volume growth condition \eqref{ineq-vg-M} and \eqref{ineq-vg-Gamma} respectively. Assume in addition that  $\mcH^{n-2}({\rm sing}M)=0$. Then there holds
		\begin{align}\label{AS16-(2.3)}
			(n-1)\int_{M}\nu\rd\mcH^{n-1}=\int_{\Gamma}\left\{\left<x,\mu\right>\nu-\left<x,\nu\right>\mu\right\}\rd\mcH^{n-2}.
		\end{align}
	\end{lemma}
	\begin{proof}

		Let $\vec{a}$ be any constant vector field in $\mbR^{n}$, and consider the following {vector field} on $M$,
		\begin{align}\label{defn-vectorfield-Y}
			Y=\left<\vec{a},\nu\right>x^T-\left<x,\nu\right>\vec{a}^T,
		\end{align}
		which is a well-defined $C^2$-vector field on ${\rm reg}M$, here $x^T=x-\left<x,\nu\right>\nu$ is the orthogonal projection of $x$ onto $T_xM$, $\vec{a}^T$ is understood similarly. Notice also that $\vert Y\vert$ is bounded on ${\rm reg}M$ by some constant $C$ since $\vec{a}$ is a constant vector field and $M$ is bounded.
		
		For $\ep>0$, we have $\varphi_\ep, S^{'}_\ep, S_\ep$ from \cref{Lem-Cut-off},
		let $\tilde{Y}_\ep:\mbR^{n}\ra\mbR^{n}$ be a $C^2$ vector field satisfying 
		\begin{align*}
			\vert\tilde{Y}_\ep\vert\leq C\text{ in a neighborhood of }M\setminus S'_\ep,\qquad\tilde{Y}_\ep=Y\text{ on }M\setminus S^{'}_{\ep}. 
		\end{align*}
		Then let $Y_\ep\in C^2(\mbR^{n};\mbR^{n})$ be a vector field satisfying
			$Y_\ep=\varphi_\ep \tilde{Y}_\ep,$
		and we readily see that
		\begin{align*}
			Y_\ep=\begin{cases}
				0\quad&\text{on }S^{'}_\ep,\\
				\varphi_\ep \tilde{Y}_\ep\quad&\text{on }S_\ep\setminus S^{'}_\ep,\\
				Y\quad&\text{on }M\setminus S_\ep.
			\end{cases}
		\end{align*}
		Notice that on $M\setminus S^{'}_\ep$, 
		\begin{align}
			&{\rm div}_M\left(x^T\right)=(n-1)-H\left<x,\nu\right>,\label{Prop2.2-eq1}\\
			&{\rm div}_M\left(\vec{a}^T\right)=-H\left<\vec a,\nu\right>\label{Prop2.2-eq2}.
		\end{align}
		Thus we have, on $M\setminus S^{'}_\ep$, there holds
		\begin{align*}
			&{\rm div}_M(Y_\ep)\\
			=&\varphi_\ep\left(\left<\vec{a},\nu\right>{\rm div}_M(x^T)+\left<\vec{a}^T,\nabla_{x^T}\nu\right>-\left<x,\nu\right>{\rm div}_M(\vec{a}^T)-\left<x^T,\nabla_{\vec{a}^T}\nu\right>\right)+\left<\nabla^M\varphi_\ep,\tilde{Y}_\ep\right>\\
			=&(n-1)\varphi_\ep\left<\vec{a},\nu\right>+\left<\nabla^M\varphi_\ep,\tilde{Y}_\ep\right>,
		\end{align*}
		where in the second equality, we have used \eqref{Prop2.2-eq1},\eqref{Prop2.2-eq2}, and the fact that $\left<\vec{a}^T,\nabla_{x^T}\nu\right>=\left<x^T,\nabla_{\vec{a}^T}N\right>=h(x^T,\vec{a}^T)$.
		
		Finally, integrating ${\rm div}_M(Y_\ep)$ on $M\setminus S'_\ep$ and using the classical divergence theorem, we have
		\begin{align}\label{smoothApproximation1}
			\int_{M}\left<\nabla^M\varphi_\ep, \tilde{Y}_\ep\right>\rd\mcH^{n-1}+\int_{M}(n-1)\varphi_\ep\left<\vec a,\nu\right>\rd\mcH^{n-1}=\int_{\Gamma}\varphi_\ep\left<Y,\mu\right>\rd\mcH^{n-2}.
		\end{align} 
		By virtue of \eqref{pw-conv} and \eqref{esti-cutoff-1} in \cref{Lem-Cut-off}, we may send $\ep\searrow0$ and use the dominated convergence theorem to get
		\begin{align*}
			(n-1)\int_M\left<\vec{a},\nu\right>\rd\mcH^{n-1}=\int_{\Gamma}\left\{\left<x,\mu\right>\left<\nu,\vec{a}\right>-\left<x,\nu\right>\left<\mu,\vec{a}\right>\right\}\rd\mcH^{n-2}.
		\end{align*}
		Since $\vec{a}$ is taken to be any constant vector field in $\mbR^{n}$, we conclude \eqref{AS16-(2.3)}.
	\end{proof}
	\begin{remark}
	\normalfont
	Note that here we use the approximation argument so that we can appeal to the classical divergence theorem.
	The reason that \cref{Lem-divergenctheorem} cannot be used  here is because the vector field $Y$ defined in \eqref{defn-vectorfield-Y} is not a globally-defined $C^2$-vector field, indeed, it is just defined on the regular part of $M$.
	\end{remark}

\section{Poincar\'e-type inequality for stable sets}\label{Sec-4}
In the spirit of Sternberg-Zumbrun \cite{SZ98}, we introduce the following admissible family of sets of finite perimeter for the study of fixed-volume variation. 

	\begin{definition}\label{admissiblesets}
		\normalfont
		For some $T>0$, a family of sets of finite perimeter in $\Om$, denoted by $\{E_t\}_{t\in(-T,T)}$, with each $E_t$ of finite perimeter and $E_0=E$, is called \textit{admissible}, if:
		\begin{enumerate}
			\item $\chi_{E_t}\ra\chi_E$ in $L^1(\Om)$ as $t\ra0$,
			\item $t\ra \mcF_\beta(E_t;\Om)$ is twice differentiable at $t=0$,
			\item $\vert E_t\vert=\vert E\vert$ for all $t\in(-T,T)$.
		\end{enumerate}
	\end{definition}
	
	The stationary and stable sets in our settings are defined in the following sense.
	
	\begin{definition}\label{stationaryandstable}
		\normalfont
		For a set of finite perimeter $E\subset\Om\subset\mbR^n$ and for an admissible family of sets $\{E_t\}_{t \in (-T,T)}$, let $\mcF_\beta(t):=\mcF_\beta(E_t)$.
		$E$ is said to be \textit{stationary for the energy functional $\mcF_\beta$ under volume constraint} if $\mcF_\beta'(0)=0$ for all admissible families $\{E_t\}$.
		A stationary set $E$ is called \textit{stable} if $\mcF_\beta''(0)\geq0$ for all admissible families $\{E_t\}$.
	\end{definition}
	\begin{proposition}\label{Prop-capillaryCMC}
		Let $E\subset\Om$ be a set of finite perimeter, which is stationary for $\mcF_\beta$ under volume constraint. Assume that $M=\overline{\p E\cap\Om}$ and $\Gamma=M\cap\p\Omega$ satisfy the local Euclidean volume growth condition \eqref{ineq-vg-M} and \eqref{ineq-vg-Gamma} respectively. Assume in addition that  $\mcH^{n-2}({\rm sing}M)=0$
    and $H\in L^1(M).$ 
		Then $E$ satisfies
		\begin{enumerate}[label=\roman*., itemsep=0pt, topsep=0pt]
			\item (CMC) On ${\rm reg}M$, the mean curvature of $M$ is constant, denoted by $H$,
			
			\item (Young's law) On ${\rm reg}M\cap\p\Om$, the measure-theoretic hypersurface $M$ intersects $\p\Om$ with a constant contact angle $\theta$ ($\cos\theta=\beta$), i.e.,
			\begin{align}\label{condi-Young-W1}
				\left<\nu,\oN\right>=-\cos\theta=-\left<\mu,\onu\right>.
			\end{align} 
		\end{enumerate}
	\end{proposition}

	\begin{proof}
 We argue as in \cite{SZ98}.
  Note that by assumptions, 
		\cref{Lem-divergenctheorem} is applicable here.
		
		{\bf Step 1. Constructing a family of admissible sets as in \cref{admissiblesets}}.
	
		We start from any variation that preserves the volume of $\Om$ at the first order when $t=0$.
		Precisely,
		let $X\in C_c^\infty(\mbR^n;\mbR^n)$ be any vector field satisfying
		\begin{align}
			\int_M\left<X,\nu\right>\rd\mcH^{n-1}(x)\label{integralofXcdotnuM}=0,\\
			X(x)\in T_x(\p\Om),\quad\forall x\in\p\Om\label{tangentalongboundary}.
		\end{align}
		By solving the Cauchy's problem:
		\begin{align}
			\frac{\p}{\p t}\Psi(t,x)=X(\Psi(t,x)),\quad &x\in\mbR^n,\label{x-flow}\\
			\Psi(0,x)=x,\quad &x\in\mbR^n,\label{3.4}
		\end{align}
		we obtain a local variation $\{\Psi_t\}_{\vert t\vert<T}$ for some small $T>0$, having $X$ as its initial velocity.
		Let $E_t:=\Psi_t(E)$, we see that $\Psi_t(\Om)\subset\Om$ by \eqref{tangentalongboundary}, and hence $E_t\subset\Om$.
		Setting $V(t):=\vert E_t\vert$,  following the same computations in the proof of \cite[Theorem 2.2]{SZ98},
		we find
	\begin{enumerate}
			\item $V'(0)=0,$
			\item $V''(0)=\int_M {\rm div}X\left<X,\nu\right>\rd\mcH^{n-1}(x).$
		\end{enumerate}
		
		Now we do some modifications inside $\Om$ to obtain a new family of admissible sets $\{\tilde{E}_t\}_{\vert t\vert<T}$, we begin by
		fixing any $x\in {\rm reg}M\cap \Om$, thanks to the regularity,
		$\p E$ can be locally written as the graph of some $C^2$-function $u_0:D'\ra\mbR^1$ ,where (up to a rotation) $D'$ is included in $R^{n-1}$ and is a neighborhood of the projection of $x$. Since $X\in C^\infty_c(\mbR^n;\mbR^n)$ satisfies \eqref{x-flow},   we can find a much smaller number, still denoted by $T$, such that not only $\p E$, but also $\p E_t$ for all $t\in(-T,T)$, can be written as a graph of a smooth function $u:D'\times(-T,T)\ra\mbR^1$ near $x$.
		
		Note that
		since $V(t)=\vert \Psi_t(E)\vert$ is second-order differentiable, and the Taylor expansion of $V(t)$ at $t=0$ is given by: $V(t)=V(0)+\frac{1}{2}t^2V''(0)+o(t^2)$,
		we can find some smooth function $g:D'\times(-T,T)\ra\mbR^1$ such that
		$g\mid_{\p D'}=0$ for any $t\in(-T,T)$,
		with
		\begin{align}\label{integrationofg}
			\int_{D'}g(x',t)\rd x'=
			\begin{cases}
				\frac{V(0)-V(t)}{t^2}\quad &t\neq0,\\
				-\frac{1}{2}V''(0) &t=0.
			\end{cases}
		\end{align}
		The new family of sets $\{\tilde{E}_t\}$ is defined via replacing the boundary portion of $\{\p E_t\}= \{\left(x',u(x',t)\right):x'\in D'\}$ by the new boundary part, denoted by $\{\p \tilde{E}_t\}$, and given by
		\begin{align*}
			\{\left(x,u(x',t)+t^2g(x',t)\right):x'\in D'\}.
		\end{align*}
		It suffices to check that such family of sets $\{\tilde E_t\}_{\vert t\vert<T}$ is admissible in the sense of \cref{admissiblesets}.
		Indeed,
		let $\tilde{V}(t)=\vert\tilde{E}_t\vert$, since $E_t$ and $\tilde{E}_t$ coincide outside $D'\times \mfR^1$ for any $t\in (-T,T)$, a direct computation gives
		\begin{align*}
			\tilde{V}(t)-V(t)
			&=\int_{D'}\left[\left(u(x',t)+t^2g(x',t)\right)-\left(u(x',t)\right)\right]\rd x'
			=t^2\int_{D'}g(x',t)\rd x'.
		\end{align*}
		Recalling \eqref{integrationofg}, we find
		\begin{align*}
			\tilde{V}(t)-V(t)=V(0)-V(t),\quad\forall t\in(-T,T),
		\end{align*}
		and it follows immediately that
		\begin{align*}
			\tilde{V}(t)=V(0),\quad\forall t\in(-T,T).
		\end{align*}
		This completes our first step.
	
		{\bf Step 2. First variation formula of the free energy functional}.
		
		For simplicity, we set $$\tilde{\mcF}_\beta(t):=P(\tilde{E}_t;\Om)-\beta P(\tilde{E}_t;\p\Om).$$
		Since $E_t$ and $\tilde{E_t}$ coincide outside $D'\times \mfR^1$ for any $t\in (-T,T)$, a simple computation gives
		\begin{align}\label{tildeF-F}
			\tilde{\mcF}_\beta(t)-\mcF_\beta(t)
			&=	\tilde{\mcF}_\beta(t)-\mcF_\beta(t)\mid_{D'}
			=\int_{D'}\left(\sqrt{1+\vert\nabla_{x'}(u+t^2g)\vert^2}-\sqrt{1+\vert\nabla_{x'}(u)\vert^2}\right)\rd x'.
		\end{align}
		Taking $t=0$ in the above equality, we find
		\begin{align*}
			\tilde{\mcF}_\beta'(0)-\mcF_\beta'(0)=0.
		\end{align*}
		The stationarity of $E$ yields that
	\begin{align}\label{eq-vanishing-firstvariation}
			\mcF_\beta'(0)=\tilde{\mcF}_\beta'(0)=0.
		\end{align}
		We need to write down the expression of $\mcF'_\beta(0)$.
		To proceed,
		notice that $E_t\subset\Om$, and hence for any open set $U$ containing $\Om$, there holds: $$P(E_t;U)=P(E_t;\Om)+P(E_t;\p\Om).$$
		Applying the first variation formula of perimeter (see e.g., \cite[Theorem 17.5]{Mag12}) and by virtue of \eqref{eq-vanishing-firstvariation},
		we thus find:
		for any $X$ satisfying \eqref{integralofXcdotnuM},\eqref{tangentalongboundary}, there holds
		\begin{align}\label{stationary1stvariation}
			\tilde\mcF_\beta'(0)=\int_{M\cap\Om} {\rm div}_M X(x)\rd\mcH^{n-1}(x)-\beta\int_{B^+}{\rm div}_{B^+}X(x)\rd\mcH^{n-1}(x)=0.
		\end{align}
		Exploiting \eqref{formu-div-M} and \eqref{formu-div-B+}, we find	\begin{align}\label{interiorandBoundaryIntegral}
			\int_{M\cap\Om}H(x)\left<X, \nu\right>(x)\rd\mcH^{n-1}(x)+\int_\Gamma \left<X,\mu-\beta\onu\right>(x)\rd\mcH^{n-2}(x)=0.
		\end{align}
		{\bf Step 3. Constant mean curvature and constant contact angle of the stationary set}.
		
		This is done by testing the first variation formula with suitable choices of vector fields.
		On the one hand,
		it is apparent that \eqref{tangentalongboundary} holds for any $X\in C^\infty_c(\Om;\mbR^n)$ satisfying \eqref{integralofXcdotnuM}.
		For any such $X$, \eqref{interiorandBoundaryIntegral} is just
		\begin{align*}
			\int_{M\cap\Om} H(x)\left<X,\nu\right>(x)\rd\mcH^{n-1}(x)=0,
		\end{align*} 
		and it follows
		that
		${\rm reg}M\cap\Om$ is of constant mean curvature. Namely, for some constant $H$, we have
		\begin{align*}
			H(x)=H,\quad \forall x\in {\rm reg}M\cap\Om.
		\end{align*}
		Back to \eqref{interiorandBoundaryIntegral}, we thus find:
		for any $X$ satisfying \eqref{integralofXcdotnuM},\eqref{tangentalongboundary},
		\begin{align}\label{constancontactangle}
			\int_\Gamma \left<X,\mu-\beta\onu\right>(x)\rd\mcH^{n-2}(x)=0.
		\end{align}
		On the other hand, we will conclude from \eqref{constancontactangle} that ${\rm reg}M$ has constant contact angle $\theta$ with $\p\Om$, where $\cos\theta=\beta$, i.e., $$\left<\nu,\oN\right>=-\cos\theta=-\left<\mu,\onu\right>\text{ on } {\rm reg}\Gamma.$$
		We begin by showing that \eqref{constancontactangle} holds for any $X_0\in C^2_c(\mbR^n;\mbR^n)$ satisfying \eqref{tangentalongboundary}.
		Indeed, for any $X_0\in C^\infty_c(\mbR^n;\mbR^n)$ satisfying \eqref{tangentalongboundary}, there exists $s>0$ and $S_0\in C^\infty_c(\Om;\mbR^n)$ such that $X:=S_0+sX_0\in C^\infty_c(\mbR^n;\mbR^n)$ satisfies \eqref{integralofXcdotnuM},\eqref{tangentalongboundary}. 
		
		Testing \eqref{constancontactangle} with $X$, we then conclude that
		\begin{align*}
			\int_\Gamma \left<X_0,\mu-\beta\onu\right>(x)\rd\mcH^{n-2}(x)=0
		\end{align*}
		holds for any $X_0\in C^\infty_c(\mbR^n;\mbR^n)$ and $X_0(x)\in T_x(\p\Om)$ for any $x\in\p\Om$.
		
		Notice that for any such $X_0$, along $\Gamma$, there holds	\begin{align*}
			\left<X_0, \mu\right>=\left<X_0,\left<\mu,\onu\right>\onu\right>,
		\end{align*}
		and hence we have: 
		\begin{align}\label{X0andConstantAngle}
			\int_\Gamma \left<X_0,\left(\left<\mu,\onu\right>-\beta\right)\onu\right>\rd\mcH^{n-2}(x)=0
		\end{align}
		for any $X_0\in C^\infty_c(\mbR^n;\mbR^n)$ satisfying \eqref{tangentalongboundary}.
		
		By virtue of the fundamental lemma of calculus of variations, we obtain: $$\left<-\nu,\oN\right>=\left<\mu,\onu\right>=\beta=\cos\theta \text{ for any }x\in {\rm reg}\Gamma.$$
		
	\end{proof}
\begin{proposition}\label{Prop-StableSet}
		Let $E\subset\Om$ be a set of finite perimeter, which is stable for $\mcF_\beta$ under volume constraint.
		Assume that $M=\overline{\p E\cap\Om}$ and $\Gamma=M\cap\p\Omega$ satisfy the local Euclidean volume growth condition \eqref{ineq-vg-M} and \eqref{ineq-vg-Gamma} respectively. Assume in addition that  $\mcH^{n-3}({\rm sing}M)=0$
    and $H\in L^1(M).$
		Then for any $C^2$-function $\zeta: {\rm reg}M\ra\mbR$ satisfying the integrability conditions:
		\begin{align}\label{condi-integrability}
		    \zeta\in L^2(M)\cap L^2(\Gamma),\quad
		    (\zeta\De_M\zeta+\vert\vert h\vert\vert^2\zeta^2)\in L^1(M),\quad
		    \left(\left<\zeta\nabla^M\zeta, \mu\right>-q\zeta^2\right)\in L^1(\Gamma)
		\end{align}
		with
		\begin{align*}
			\int_{{\rm reg}M}\zeta(x)\rd\mcH^{n-1}(x)=0,
		\end{align*}
		the following Poincar\'e-type inequality holds:
		\begin{align}\label{ineq-Poincare}
			J(\zeta):=-\int_{M\cap\Om}\left(\zeta\Delta_M\zeta+\vert\vert h\vert\vert^2\zeta^2\right)\rd\mcH^{n-1}(x)
			+\int_{\Gamma}\left(\left<\zeta\nabla^M\zeta, \mu\right>-q\zeta^2\right)\rd\mcH^{n-2}\geq0,
		\end{align}
		where
		\begin{align}\label{q}
	q=\frac{1}{\sin\theta}h^{\p\Om}( \onu,\onu)+\cot\theta h(\mu,\mu ).
		\end{align}
\end{proposition}

To derive the second variation formula, we need the following classical computations that are carried out on ${\rm reg}M$.
	\begin{lemma}[\cite{RS97}, Lemma 4.1]\label{Lem-RS97}
		Let $E\subset\Om$ be as in \cref{Prop-StableSet} and $\Psi_t$ be a $C^2$-variation whose initial velocity $X:=\frac{\p}{\p t}\mid_{t=0}\Psi_t$ satisfies \eqref{integralofXcdotnuM} and \eqref{tangentalongboundary}.
		Let $X_t(x):=\frac{\p}{\p s}\mid_{s=t}\Psi_s(x)$ denote the velocity of the variation at $t$. 
		Let $\nabla^M$, $\tilde{\nabla}$ denote the gradient on ${\rm reg}M, {\rm reg}\Gamma$, respectively,
		and $X_M^T$ (resp. $X_\Gamma^T$) the tangential part of $X$ with respect to $M$ (resp. to $\Gamma$).
		Let also $S_0,S_1,S_2$ denote respectively the classical shape operator in differential geometry, of ${\rm reg}M$ in $\mbR^n$ with respect to $-\nu$, of ${\rm reg}\Gamma$ in $M$ with respect to $\mu$ and of ${\rm reg}\Gamma$ in $\p B^{+}$ with respect to $\onu$.
		Let $f$ be the $C^2$-function defined on ${\rm reg}M$ by $f=\left<X,-\nu\right>$, then
		on ${\rm reg}M$, there holds:
		\begin{enumerate}
			\item $(-\nu)'=-\nabla^M f-S_0(X_M^T)$,
			\item $(\mu)'=\left(\frac{\p f}{\p\mu}+h(X_M^T,\mu)\right)(-\nu)+fS_0(\mu)-fh\left(\mu,\mu\right)\mu-S_1(X_\Gamma^T)+\cot\theta\nabla^{B^+}f$,
			\item $(\onu)'=-h^{\p\Om}(X,\onu)\oN-S_2(X_\Gamma^T)+\frac{1}{\sin\theta}\nabla^{B^+}f$,
			\item\label{4-0} 					$\left<X',\mu-\beta\onu\right>+\left<X,(\mu)'-(\beta\onu)'\right>=
			f\frac{\p f}{\p \mu}-qf^2$,
		\end{enumerate}
		where $q$ is given by \eqref{q}.
		Here we denote by a "prime" the first derivative $\frac{\p}{\p t}\mid_{t=0}$ in the Euclidean space $\mbR^n$.
	\end{lemma}
	
	\begin{proof}[Proof of \cref{Prop-StableSet}]
		
		Consider any $C^2$-function $\zeta:{\rm reg}M\ra\mbR^1$ that has the desired integrability and satisfies $\int_M\zeta \rd\mcH^{n-1}(x)=0$ .
		For any small $\ep>0$, we have $\varphi_\ep,S'_\ep$ and $S_\ep$ from \cref{Lem-Cut-off} (notice that this is the case when $q=2$).
		
		First, we consider a $C^2$-extension of $\zeta$ from $M\setminus S_\ep'$ to $\mfR^n$ (still denoted by $\zeta$),
		and set $\tilde\zeta_\ep:=\varphi_\ep\cdot\zeta$.
		By virtue of \cref{Lem-Cut-off}, we claim that
		\begin{enumerate}
			\item $\tilde{\zeta}_\epsilon\equiv0$ on $S_\epsilon'$,
			\item $\tilde{\zeta}_\epsilon\equiv \zeta$ on $M\setminus S_\epsilon$,
			\item $\left(\tilde\zeta_\ep\Delta_M\tilde\zeta_\ep+\vert\vert h\vert\vert^2\tilde\zeta_\ep^2\right)\ra\left(\zeta\Delta_M\zeta+\vert\vert h\vert\vert^2\zeta^2\right)$ in $L^1(M)$ as $\ep\searrow0$
			\item $\left(\left<\tilde\zeta_\ep\nabla^M\tilde\zeta_\ep, \mu\right>-q\tilde\zeta_\ep^2\right)\ra\left(\left<\zeta\nabla^M\zeta, \mu\right>-q\zeta^2\right)$ in $L^1(\Gamma)$ as $\epsilon\searrow0$.
		
		\end{enumerate}
		 (1)(2) are obvious. For (3), we see that
		\begin{align*}
		  & \int_M \left(\tilde\zeta_\ep\Delta_M\tilde\zeta_\ep+\vert\vert h\vert\vert^2\tilde\zeta_\ep^2\right) \rd\mcH^{n-1}\\
    =&\int_M\varphi_\ep^2\left(\zeta\De_M\zeta+\vert\vert h\vert\vert^2\zeta^2\right)+\left<\nabla^M\varphi_\ep^2,\zeta\nabla^M\zeta\right>+\zeta^2\varphi_\ep\De_M\varphi_\ep\rd\mcH^{n-1}\\
     =&\int_M\varphi_\ep^2\left(\zeta\De_M\zeta+\vert\vert h\vert\vert^2\zeta^2\right)-\zeta^2\vert\nabla^M\varphi_\ep\vert^2\rd\mcH^{n-1}+\int_\Gamma\zeta^2\varphi_\ep\left<\nabla^M\varphi_\ep,\mu\right>\rd\mcH^{n-2}.
		\end{align*}
		Using \cref{Lem-Cut-off} and the integrability assumptions, we get (3) from the dominated convergence theorem. In a similar way, we can get
		(4).
		
		 Recall that $S_\epsilon \subset\bigcup_{i=1}^{N_1}(B_{(2+\alpha)r_i}(z_i))$ for an arbitrary small $\alpha>0$, as shown in the proof of \cref{Lem-Cut-off}.
		 Using the volume growth condition \eqref{ineq-vg-M}, we find
		\begin{align*}
			\int_M\left(\zeta-\tilde{\zeta}_\epsilon\right) \rd\mcH^{n-1}
			=\int_{M\cap S_\epsilon}(1-\varphi_\epsilon)\zeta\rd\mcH^{n-1}
			\leq(\int_{M\cap S_\ep}1\rd\mcH^{n-1})^{1/2}\vert\vert\zeta\vert\vert_{L^2(M)}\leq C\ep^{1/2}.
		\end{align*}
		
		With this estimate, we can modify $\tilde{\zeta}_\epsilon$ on $M\setminus S_\epsilon$, and obtain a new $C^2$-function, denoted by $\zeta_\epsilon$, with the following properties: 
		\begin{enumerate}
			\item $\zeta_\epsilon\equiv0$ on $S_\epsilon'$,
				\item $\int_M \zeta_\epsilon d\mcH^{n-1}=0$.
			\item $\left(\zeta_\ep\Delta_M\zeta_\ep+\vert\vert h\vert\vert^2\zeta_\ep^2\right)\ra\left(\zeta\Delta_M\zeta+\vert\vert h\vert\vert^2\zeta^2\right)$ in $L^1(M)$ as $\ep\searrow0$
			\item $\left(\zeta_\ep\left<\nabla^M\zeta_\ep, \mu\right>-q\zeta_\ep^2\right)\ra\left(\left<\zeta\nabla^M\zeta, \mu\right>-q\zeta^2\right)$ in $L^1(\Gamma)$ as $\epsilon\searrow0$.
		\end{enumerate}
		Precisely,
		we may find a smooth function $\eta:\mbR^n\ra\mbR^1$ s.t.,
		\begin{enumerate}
			\item $\int_M\eta \rd\mcH^{n-1}=1$,
			\item $\eta\equiv0$ on $S_\epsilon'$.
		\end{enumerate}
		By setting ${\rm err}(\epsilon)=\int_M(\zeta-\tilde{\zeta}_\epsilon) \rd\mcH^{n-1}$, and $\eta_\epsilon={\rm err}(\epsilon)\eta$, we get
		\begin{enumerate}
			\item $\int_M\eta_\epsilon \rd\mcH^{n-1}=\int_M {\rm err}(\epsilon)\eta \rd\mcH^{n-1}={\rm err}(\epsilon)$,
			\item $\eta_\epsilon\equiv0$ on $S_\epsilon'$,
			\item $\vert\eta_\epsilon\vert_{L^\infty(M)}\leq \vert {\rm err}(\epsilon)\vert\cdot \sup_M \vert\eta\vert\leq C\epsilon^{1/2}$,
			\item
			$\vert\nabla\eta_\ep\vert_{L^\infty(M)}\leq\vert{\rm err}(\ep)\vert\cdot\sup_M\vert\nabla\eta\vert\leq C\ep^{1/2}$,
			\item
			$\vert\nabla^2\eta_\ep\vert_{L^\infty(M)}\leq\vert{\rm err}(\ep)\vert\cdot\sup_M\vert\nabla^2\eta\vert\leq C\ep^{1/2}$.
		\end{enumerate}
		Let $\zeta_\epsilon:=\tilde{\zeta}_\epsilon+\eta_\epsilon$.
  Notice that $|\De_M\eta_\ep|\le |\n^2 \eta|+|H||\n\eta|$, by the assumption $H\in L^1(M)$, we see $\|\De_M\eta_\ep\|_{L^1(M)}\le C\ep^{\frac12}$. It is easy to check $\zeta_\epsilon$ satisfies all the desired properties.
		
	Since ${\rm sing}M\subset S'_\ep$, we may find some vector field $\nu_\epsilon\in C^2_c(\mbR^n;\mbR^n)$ satisfying
		\begin{enumerate}
			\item $\vert \nu_\epsilon\vert=1$ in a neighborhood of $M\setminus S_\epsilon'$,
			\item $\nu_\epsilon=\nu$ on $M\setminus S'_\epsilon$,
		\end{enumerate}
		and some vector field $\mu_\epsilon\in C^2_c(\mbR^n;\mbR^n)$ satisfying 
		\begin{enumerate}
			\item $\vert \mu_\epsilon\vert=1$ in a neighborhood of $(M\setminus S_\epsilon')\cap\p\Om$,
			\item $\mu_\ep(x)\in T_xM$ on $M\setminus S'_\ep$,
			\item $\mu_\epsilon=\mu$ on $(M\setminus S'_\epsilon)\cap\p\Om$,
		\end{enumerate}
		such that 
		the vector field $X_\ep:=-\zeta_\epsilon\left(\frac{\beta}{\sqrt{1-\beta^2}}\mu_\ep
			+\nu_\epsilon\right)$ satisfies:
			\begin{align*}
			    X_\epsilon(x)\in T_x(\p\Om)\quad\text{for all }x\in\p\Om,
			\end{align*}
		
		
		Notice that such $X_\epsilon$ exists since these conditions can both be satisfied by virtue of the fact that $M$ intersects $\p\Om$ with the constant contact angle $\theta$, where $\cos\theta=\beta$ and $\frac{\beta}{\sqrt{1-\beta^2}}=\cot{\theta}$, and hence on ${\rm reg}M\cap\p\Om$, $\frac{\beta}{\sqrt{1-\beta^2}}\mu_\epsilon(x)
		+\nu_\epsilon(x)\in T_x\p\Om$.
		
		A direct computation then gives
		\begin{align*}
			\int_M \left<X_\epsilon,-\nu\right>\rd\mcH^{n-1}
			=\frac{\beta}{\sqrt{1-\beta^2}}\int_M\zeta_\epsilon \left<\mu_\epsilon,\nu\right>\rd\mcH^{n-1}+\int_M\zeta_\epsilon\left<\nu_\epsilon,\nu\right>\rd\mcH^{n-1}
			=0,
		\end{align*}
		which shows that $X_\epsilon$ satisfies \eqref{integralofXcdotnuM}, \eqref{tangentalongboundary}.
		Let $\Psi_t^\ep$ denotes the $C^2$-local variation induced by $X_\ep$.
		Following the same argument as in the proof of \cref{Prop-capillaryCMC}, we  obtain an admissible family of sets $\{\tilde{E}^\ep_t\}$ by a smooth modification through the graph function (denoted by $g_\epsilon$) inside $\Om$.
		By virtue of \eqref{integrationofg} and \eqref{tildeF-F} and the fact that $E$ is stable, we get
		\begin{align}\label{condi-stable}
		    0\leq\frac{\rd^2}{\rd t^2}\mid_{t=0}\mcF_\beta(\Psi_t^\ep(E))-H\frac{\rd^2}{\rd t^2}\mid_{t=0}V(\Psi_t^\ep(E)).
		\end{align}
		
		Recall that for any $\ep>0$, we have $\varphi_\ep,S_\ep,S'_\ep$ from \cref{Lem-Cut-off}.
		Notice that since $\zeta_\ep\equiv0$ on $S'_\ep$, and hence instead of $\Psi_t^\ep(E)$, it suffice to consider the following sets when dealing with derivatives: $\Psi_t^\ep(E\setminus S'_\ep)$, $M^\ep_t:=\Psi_t^\ep(M\setminus S_\ep')$, $B^\ep_t:=\Psi_t^\ep(B^+\setminus S'_\ep)$. We denote by $A^\ep(t)$ the area of $M_t^\ep$ and $B^\ep(t)$ the area of $B^t_\ep$.
		
		Note that these sets are smooth enough since ${\rm sing}M\subset S'_\ep$, so that we can use the classical divergence theorem and tangential divergence theorem in the following.
		
		Let $\mcF_\beta^\ep(t):=A^\ep(t)-\beta B^\ep(t)$, our aim is to derive the explicit form of $\frac{\rd^2}{\rd t^2}\mid_{t=0}\mcF^\ep_\beta(0)$ and then appeal to the stable condition
		
		To this end, we first observe that for small enough time $t$, $\Psi_t^\ep(M\setminus S'_\ep)$ is regular enough and hence
		we can use the classical divergence theorem and the area formula to see that
		\begin{align*}
		    \frac{\rd}{\rd t}A^\ep(t)
		=&\int_{\Psi^\ep_t(M\setminus S'_\ep)}{\rm div}_{M^\ep_t}X_\ep\mid_y\rd\mcH^{n-1}(y)
		    \\=&\int_{\Psi^\ep_t(M\setminus S'_\ep)}H(t)\left<X_\ep,\nu_t\right>\mid_y\rd\mcH^{n-1}(y)
		    +\int_{\Psi^\ep_t(\Gamma\setminus S'_\ep)}\left<X_\ep,\mu_t\right>\mid_y\rd\mcH^{n-2}(y)
      \\=&\int_{M\setminus S'_\ep}H(t)\left<X_\ep,\nu_t\right>\mid_{\Psi^\ep_t(x)}{\rm J}^{{\rm reg}M}\Psi^\ep_t(x)\rd\mcH^{n-1}(x)\\&+\int_{\Gamma\setminus S_\ep'}\left<X_\ep,\mu_t\right>\mid_{\Psi^\ep_t(x)}{\rm J}^{{\rm reg}M}\Psi^\ep_t(x)\rd\mcH^{n-2}(x).
		\end{align*}
		where $\nu_t$ denotes the unit outer unit normal of $\Psi_t^\ep(M\setminus S'_\ep)$ and $\mu_t,\bar\nu_t$ are understood in a similar way.
		Similarly, we get
		\begin{align*}
		    \frac{\rd}{\rd t}B^\ep(t)
		    =\int_{\Gamma\setminus S'_\ep}\left<X_\ep,\bar\nu_t\right>\mid_{\Psi^\ep_t(x)}{\rm J}^{{\rm reg}\Gamma}\Psi^\ep_t(x)\rd\mcH^{n-2}(x),
		\end{align*}
		and hence
		\begin{align*}
		    \frac{\rd}{\rd t}\mcF_\beta^\ep(t)
		    =\int_{M\setminus S'_\ep}H(t)\left<X_\ep,N_t\right>\mid_{\Psi^\ep_t(x)}{\rm J}^{{\rm reg}M}\Psi^\ep_t(x)\rd\mcH^{n-1}(x)
		    \\+\int_{\Gamma\setminus S'_\ep}\left<X_\ep,\mu_t-\beta\bar\nu_t\right>\mid_{\Psi^\ep_t(x)}{\rm J}^{{\rm reg}\Gamma}\Psi^\ep_t(x)\rd\mcH^{n-2}(x).
		\end{align*}
		Since $M\setminus S'_\ep$ and $\Gamma\setminus S'_\ep$ are smooth enough
		and ${\rm reg}M$ is of constant mean curvature by \cref{Prop-capillaryCMC}, we can further differentiate the above equation and evaluate at $t=0$ to obtain
		\begin{align}\label{F''beta0}
				\frac{\rd^2}{\rd t^2}\mid_{t=0}\mcF^\ep_\beta(t)
				&=\int_{M\setminus S'_\ep} H'(0)\left<X_\ep,\nu\right>\mid_x\rd\mcH^{n-1}(x)\notag\\
				&+H\frac{\rd}{\rdt}\mid_{t=0}\left(\int_{M\setminus S_\ep'}\left<X_\ep,\nu_t\right>\mid_{\Psi^\ep_t(x)}{\rm J}^{{\rm reg}M}\Psi^\ep_t(x)\rd\mcH^{n-1}(x)\right)\notag\\
				&+\int_{\Gamma\setminus S'_\ep}\left<\frac{\p}{\p t}\mid_{t=0}X_\ep\left(\Psi^\ep_t(x)\right),\mu-\beta\onu\mid_x\right>\rd\mcH^{n-2}(x)\\
				&+\int_{\Gamma\setminus S'_\ep} \left<X^\ep_\ep(x)
				,\left(\frac{\p}{\p t}\mid_{t=0}\left(\mu_t-\beta\onu_t\mid_{\Psi^\ep_t(x)}\right)\right)\right>\rd\mcH^{n-2}(x)\notag\\
				&+\int_{\Gamma\setminus S'_\ep} \left<X_\ep
				,\mu-\beta\onu\mid_x\right>\frac{\p}{\p t}\mid_{t=0}{\rm J}^{{\rm reg}\Gamma}\Psi^\ep_t(x)\rd\mcH^{n-2}(x).\notag
			\end{align}
			
			For the first term in \eqref{F''beta0}, by the evolution equation (c.f., \cite[(4.1b)]{Rosenberg93})
			\begin{align*}
				H'(0)=\De_M \zeta_\ep+\vert \vert h\vert\vert^2\zeta_\ep,
			\end{align*}
			we have
			\begin{align*}
				&\int_{M\setminus S'_\ep} H'(0)\left<X_\ep,\nu\right>\mid_x\rd\mcH^{n-1}(x)
				=-\int_{M\setminus S'_\ep} \left(\zeta_\ep\De_M f+\vert \vert h\vert\vert^2\zeta_\ep^2\right)\mid_x \rd\mcH^{n-1}(x).
			\end{align*}
			
			For the second term in \eqref{F''beta0}, notice that 
			\begin{align*}
				\frac{\rd}{\rd t}V(\Psi_t^\ep(E\setminus S'_\ep))
				&=\int_{\Psi_t^\ep(E\setminus S'_\ep)}{\rm div}X_\ep\mid_{y}\rd\mcH^n(y)
				=\int_{\Psi_t^\ep(M\setminus S'_\ep)}\left<X_\ep, \nu_t\right>\mid_y \rd\mcH^{n-1}(y)\\
				&=\int_{M\setminus S'_\ep}\left<X_\ep,\nu_t\right>\mid_{\Psi^\ep_t(x)}{\rm J}^{{\rm reg}M}\Psi^\ep_t(x)\rd\mcH^{n-1}(x)
			\end{align*}
			where we have used the fact that $\zeta_\ep\equiv0$ on $\p S'_\ep$ for the second equality, and the area formula in the last equality.
			In particular, this gives
			\begin{align*}
				H\frac{\rd^2}{\rd t^2}\mid_{t=0}V(\Psi_t^\ep(E\setminus S'_\ep))&=H\frac{\rd}{\rdt}\mid_{t=0}\left(\int_{M\setminus S'_\ep}\left<X_\ep,\nu_t\right>\mid_{\Psi^\ep_t(x)}{\rm J}^{{\rm reg}M}\Psi^\ep_t(x)\rd\mcH^{n-1}(x)\right).
			\end{align*}
			
			The fifth term is vanishing,
			due to the fact that
			$X_\ep(x)\in T_x\p\Om$ and observe that on $\Gamma\setminus S'_\ep$, by \cref{Prop-capillaryCMC}, there holds $$(\mu-\beta\onu)\mid_x\perp T_x\p\Om.$$
			Combining these equalities, we find	\begin{align*}
			\frac{\rd^2}{\rd t^2}\mid_{t=0}\mcF^\ep_\beta(t)
				=&-\int_{M\setminus S'_\ep} \left(\zeta_\ep\De_M \zeta_\ep+\vert \vert h\vert\vert^2\zeta_\ep^2\right) \rd\mcH^{n-1}(x)\notag\\
				&+H\frac{\rd^2}{\rd t^2}\mid_{t=0}V(\Psi_t^\ep(E\setminus S'_\ep))\notag\\
				&+\int_{\Gamma\setminus S'_\ep}\left<\frac{\p}{\p t}\mid_{t=0}X_\ep\left(\Psi^\ep_t(x)\right),\mu-\beta\onu\mid_x\right>\rd\mcH^{n-2}(x)\\
				&+\int_{\Gamma\setminus S'_\ep} \left<X_\ep(x)
				,\left(\frac{\p}{\p t}\mid_{t=0}\left(\mu_t-\beta\onu_t\mid_{\Psi^\ep_t(x)}\right)\right)\right>\rd\mcH^{n-2}(x)\notag.
			\end{align*}
		Taking also \cref{Lem-RS97} \eqref{4-0} into account,
		we conclude
		\begin{align*}
					\frac{\rd^2}{\rd t^2}\mid_{t=0}\mcF^\ep_\beta(t)
				=&-\int_{M\setminus S'_\ep} (\zeta_\ep\De_M \zeta_\ep+\vert \vert h\vert\vert^2\zeta^2)\rd\mcH^{n-1}(x)\notag\\
				&+H\frac{\rd^2}{\rd t^2}\mid_{t=0}V(\Psi_t^\ep(E\setminus S'_\ep))
				+\int_{\Gamma\setminus S'_\ep}\left(\zeta_\ep\frac{\p \zeta_\ep}{\p\mu}-q\zeta_\ep^2\right)\rd\mcH^{n-2}.
			\end{align*}
Back to \eqref{condi-stable}, we thus obtain
			\begin{align*}
  & -\int_{M} (\zeta_\ep\De_M \zeta_\ep+\vert \vert h\vert\vert^2\zeta_\ep^2)\rd\mcH^{n-1}(x)
				+\int_{\Gamma}\left(\zeta_\ep\frac{\p \zeta_\ep}{\p\mu}-q\zeta_\ep^2\right)\rd\mcH^{n-2}
    \\=&
			    -\int_{M\setminus S'_\ep} (\zeta_\ep\De_M \zeta_\ep+\vert \vert h\vert\vert^2\zeta_\ep^2)\rd\mcH^{n-1}(x)
				+\int_{\Gamma\setminus S'_\ep}\left(\zeta_\ep\frac{\p \zeta_\ep}{\p\mu}-q\zeta_\ep^2\right)\rd\mcH^{n-2}\geq0.
			\end{align*}
			Finally, by virtue of the fact that $\left(\zeta_\ep\Delta_M\zeta_\ep+\vert\vert h\vert\vert^2\zeta_\ep^2\right)\ra\left(\zeta\Delta_M\zeta+\vert\vert h\vert\vert^2\zeta^2\right)$ in $L^1(M)$ and $\left(\zeta_\ep\left<\nabla^M\zeta_\ep, \mu\right>-q\zeta_\ep^2\right)\ra\left(\left<\zeta\nabla^M\zeta, \mu\right>-q\zeta^2\right)$ in $L^1(\Gamma)$ as $\ep\searrow0$, we may send $\ep\searrow0$ to conclude that \eqref{ineq-Poincare} holds.
			
	\end{proof}

We end this section with the following remark.
\begin{remark}\label{rmk-local-min}
\normalfont
A local minimizer of the free energy functional under volume constraint, is clearly a stationary and stable set in the sense of \cref{stationaryandstable}. 
By virtue of \cref{regularity-thm} and \cref{eucl-vol-growth}, we know that 
a local minimizer satisfies all the conditions in \cref{Prop-capillaryCMC} and \cref{Prop-StableSet}.
\end{remark} 
	

	\section{Volume-constraint local minimizers in a wedge-shaped domain}\label{Sec-5}

In this section, we prove \cref{umbillic} for the case $\O=\mbR^n_+$. In fact, we shall study a more general setting that $\O$ is a wedge-shaped domain.

We first clarify the terminologies regarding the so-called wedge-shaped domain.
	Let $\mfW$ be an unbounded domain in $\mbR^{n}$($n\geq3$), which is determined by a finite family of mutually intersecting hyperplanes $P_1,\ldots, P_L$, for some integer $L\geq1$.
	Up to a translation, we may assume that the origin $O\in\mbR^{n}$ is in the intersection $\bigcap_{i=1}^LP_i$.
	We denote by $\p\mfW$ its boundary.
	Let $\mfn_1,\ldots,\mfn_L$ be the exterior unit normal to $P_i$ in $\mfW$. We call such $\mfW$ a \textit{wedge-shaped domain} (in the literature \cite{LX17,Souam21}, such domains are called \textit{domains with planar boundaries}) when $\mfW$ satisfies that: $\{\mfn_{1},\ldots,\mfn_{L}\}$ are linearly independent. In the special case $L=1$, $W$ is a half-space.
	
	Let $E\subset\mfW$ be a set of finite perimeter, for simplicity we assume $P(E;P_i)>0$ for $i=1,\ldots,L$.
    Let $M=\overline{\p E\cap\mfW}$.
	In all follows, we assume that $M$ is disjoint from the edges $P_i\cap P_j$ for any $(i\neq j)$ and $\overline{E}$ is away from the edges.
Let $B_i^+$ denote the set $(\p E\setminus M)\cap P_i$, which is relatively open in $\p\mfW$ and smooth; let $\Gamma_i$ denote the closed set $M\cap P_i$.
Let $\nu,\oN_i$ denote the outwards pointing unit normals of $M, B_i^+$, respectively, when they exist; $\mu_i,\onu_i$ denote the exterior unit conormals of $\Gamma_i$ in $M,B_i^+$, respectively. See \cref{figure2} for illustration.
	
In this situation, the free energy functional $\mcF_{L}(E;\mfW)$ is given by
\begin{align}\label{FreeEnergyFunctional-Poly1}
	\mcF_L(E;\mfW)=P(E;\mfW)-\sum_{i=1}^L\beta_iP(E;P_i),
\end{align} 
where for each $i$, $\beta_i\in(-1,1)$ is a prescribed constant determining the contact angles.
Let $\mfk$ be a constant vector defined by
\begin{align}\label{mfk}
 \mfk=\sum_{i=1}^Lc_i\mfn_i,
\end{align}
and the constants $c_i$ are such that $\left<\mfk,\mfn_i\right>=\beta_i$.
We refer the interested readers to \cite[(1.7)]{JWXZ22} for the geometric meaning of $\mfk$. 
	
	\begin{figure}[h]
	\centering
	\includegraphics[height=9cm,width=15cm]{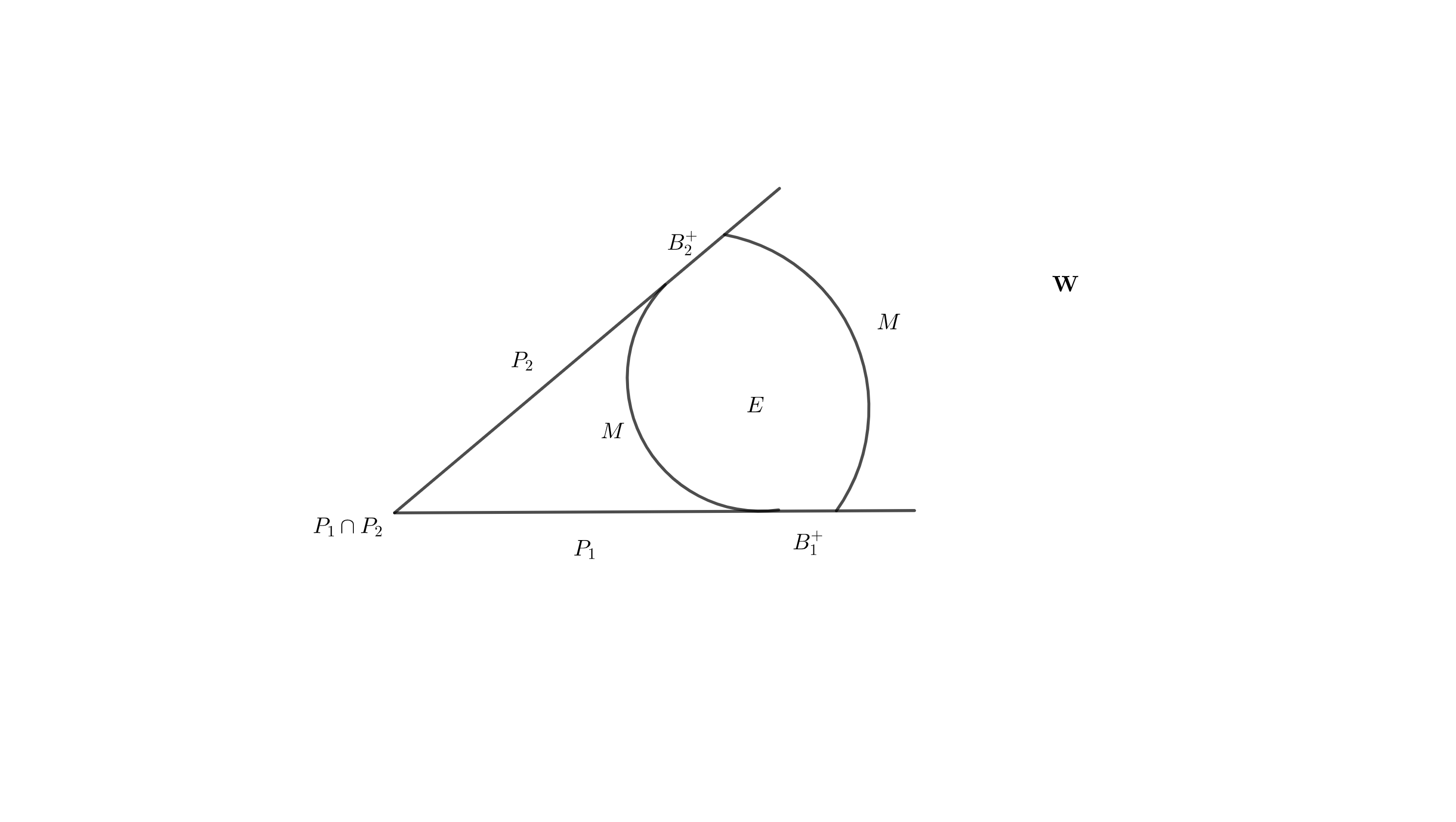}
	\caption{Notations for wedge}
	\label{figure2}
\end{figure}

 	In this section, we extend the rigidity results for  smooth stable capillary hypersurfaces in a wedge-shaped domain and in a half-space to the singular setting, in light of the arguments in \cite{LX17,Souam21}. 
	\begin{theorem}\label{Thm-wedgeShaped}
		Let $\mfW\subset\mbR^n$ be a wedge-shaped domain with planar boundaries $P_1,\ldots, P_L$. 
		Let $E\subset\mfW$ be a set of finite perimeter
		whose closure is away from the edges of $\mfW$, which is a local minimizer of the free energy functional \eqref{FreeEnergyFunctional-Poly1} under volume constraint among sets of finite perimeter. 
		Assume $\vert\mfk\vert\leq1$. Then $M=\overline{\p E\cap\mfW}$ is part of a sphere that intersects $P_i$ at the contact angle $\arccos\beta_i$.
		
		In particular, the conclusion is true for $\mfW=\mathbf
		{R}^n_+$.
	\end{theorem}

	Let us begin by noticing that with slight modifications, one can recover the proof of \cref{Prop-capillaryCMC} and \cref{Prop-StableSet} for stationary and stable sets (in the sense of \cref{stationaryandstable}) in the wedge-shaped domain.
	Here we state them without proof.
	\begin{proposition}\label{Prop-capillaryCMC-W}
	    	Let $E\subset\mfW$ be a set of finite perimeter whose closure is away from the edges of $\mfW$, which is stationary for $\mcF_L$ under volume constraint. Assume that $M$ and $\Gamma_i$ satisfy the local Euclidean volume growth condition \eqref{ineq-vg-M} and \eqref{ineq-vg-Gamma} respectively. Assume in addition that  $\mcH^{n-2}({\rm sing}M)=0$
    and $H\in L^1(M).$
		Then $E$ satisfies
		\begin{enumerate}[label=\roman*., itemsep=0pt, topsep=0pt]
			\item (CMC) On ${\rm reg}M$, the mean curvature of $M$ is constant, denoted by $H$,
			
			\item (Young's law) On ${\rm reg}M\cap P_i$, the measure-theoretic hypersurface $M$ intersects $\p\Om$ with a constant contact angle $\theta_i$ ($\cos\theta_i=\beta_i$), i.e.,
			\begin{align}\label{condi-Young-W}
				\left<\nu,\oN_i\right>=-\cos\theta_i=-\left<\mu,\onu_i\right>.
			\end{align} 
		\end{enumerate}
	\end{proposition}
	
	\begin{proposition}\label{Prop-StableSet-W}
		
	    	Let $E\subset\mfW$ be a set of finite perimeter whose closure is away from the edges of $\mfW$, which is stable for $\mcF_L$ under volume constraint.
	    	Assume that $M$ and $\Gamma_i$ satisfy the local Euclidean volume growth condition \eqref{ineq-vg-M} and \eqref{ineq-vg-Gamma} respectively. Assume in addition that  $\mcH^{n-3}({\rm sing}M)=0$
    and $H\in L^1(M).$
		Then for any $C^2$-function $\zeta: {\rm reg}M\ra\mbR$ satisfying the integrability conditions
		\begin{align}\label{condi-integrability-W}
		    \zeta\in L^2(M)\cap L^2(\Gamma_i),\quad
		    (\zeta\De_M\zeta+\vert\vert h\vert\vert^2\zeta^2)\in L^1(M),\quad
		    \left(\left<\zeta\nabla^M\zeta, \mu_i\right>-q_i\zeta^2\right)\in L^1(\Gamma),
		\end{align}
		with
		\begin{align*}
			\int_{{\rm reg}M}\zeta(x)\rd\mcH^{n-1}(x)=0,
		\end{align*}
		the following Poincar\'e-type inequality holds:
		\begin{align}\label{result2-wedge}
			-\int_{M\cap\mfW}\left(\zeta\Delta_M\zeta+\vert\vert h\vert\vert^2\zeta^2\right)\rd\mcH^{n-1}(x)
			+\sum_{i=1}^L\int_{\Gamma_i}\left(\zeta\left<\nabla^M\zeta, \mu_i\right>-q_i\zeta^2\right)\rd\mcH^{n-2}\geq0,
		\end{align}
		Here $\nabla^M,\De_M$ denote the tangential gradient and tangential Laplacian with respect to $M$, and
		\footnote{Notice that the boundaries of $\mfW$ are planar, thus $h^{P_i}\equiv0$.}
		\begin{align}\label{q-wedge}
			q_i=\cot\theta_i h(\mu_i,\mu_i).
		\end{align}
\end{proposition}

	Exploiting \cref{PropAS16-2.4}, we obtain the following Minkowski-type formula for singular hypersurfaces. The smooth case has been derived in \cite[Lemma 5]{LX17}.
	\begin{proposition}[Minkowski-type formula in a Wedge-shaped domain]\label{PropMinkowski-Poly}
		Let $E\subset\mfW$ be a set of finite perimeter whose closure is away from the edges of $\mfW$, which is stationary for $\mcF_L$ under volume constraint. Assume that $M$ and $\Gamma_i$ satisfy the local Euclidean volume growth condition \eqref{ineq-vg-M} and \eqref{ineq-vg-Gamma} respectively. Assume in addition that  $\mcH^{n-2}({\rm sing}M)=0$
    and $H\in L^1(M)$,
        then there holds
		\begin{align}\label{Minkowski-type-formula}
			\int_M\left\{(n-1)-H\left<x,\nu\right>+(n-1)\left<\nu,\mfk\right>\right\}\rd\mcH^{n-1}=0.
		\end{align}

	\end{proposition}
	\begin{proof}
		For each $\ep>0$, we have $\varphi_\ep, S^{'}_\ep, S_\ep$ from \cref{Lem-Cut-off}. Let $\tilde{Y}_\ep:\mbR^{n}\ra\mbR^{n}$ be a $C^2$-vector field satisfying 
		\begin{align*}
			\vert\tilde{Y}_\ep\vert\leq C\text{ in a neighborhood of }M\setminus S'_\ep,\qquad\tilde{Y}_\ep=\left(x-\left<x,\nu\right>\nu\right)\text{ on }M\setminus S^{'}_{\ep}. 
		\end{align*}
		Notice that $M$ is bounded, and hence $\vert x-\left<x,\nu\right>\nu\vert$ is a bounded function on ${\rm reg}M$.
		Let $Y_\ep\in C^2(\mbR^{n};\mbR^{n})$ be the vector field given by $Y_\ep=\varphi_\ep \tilde{Y}_\ep.$
		
		Now we argue as the proof of \cref{PropAS16-2.4}.
		Integrating ${\rm div}_M\left(Y_\ep\right)$ over $M\setminus S'_\ep$ and using the tangential divergence theorem, then sending $\ep\searrow0$.
		By virtue of \eqref{pw-conv},\eqref{esti-cutoff-1}, and the fact that $H\in L^1(M)$, we can use the dominated convergence theorem to see that
		\begin{align}\label{AS16-5.1-Wedge}
			\sum_{i=1}^L\int_{\Gamma_i}\left<x,\mu_i\right>\rd\mcH^{n-2}=\int_M\left\{(n-1)-H\left<x,\nu\right>\right\}\rd\mcH^{n-1}.
		\end{align}
		
		On the other hand,
  by virtue of \cref{Prop-capillaryCMC-W} (the constant contact angle condition), on ${\rm reg}\Gamma_i$ there holds
		\begin{align}\label{ContactAngle-Wedge}
			-\cos{\theta_i}\nu+\sin{\theta_i}\mu_i=\oN_i,
		\end{align}
		it follows that
		\begin{align}\label{x,n_i}
			-\cos\theta_i\left<x,\nu\right>
			+\sin\theta_i\left<x,\mu_i\right>=\left<x,\mfn_i\right>,
		\end{align}
		and it follows from $\left<x,\bar N_i\right>=0$ on each $P_i$ (since the origin $O\in\mfW$) that
		\begin{align}\label{eq-Minko-W-1}
		    \cos\theta_i\left<x,\nu\right>=\sin\theta_i\left<x,\mu_i\right>.
		\end{align}
		Notice that \eqref{mfk} implies
		\begin{align}\label{k,n_i}
			-\cos\theta_i\left<\nu,\mfk\right>+\sin\theta_i\left<\mu_i,\mfk\right>=\left<\mfn_i,\mfk\right>=\cos\theta_i.
		\end{align}
		Interior producting \eqref{AS16-(2.3)} with $\mfk$, exploiting \eqref{x,n_i} and \eqref{k,n_i}, we obtain
		\begin{align}\label{AS16-5.2-Wedge}
			(n-1)\int_{M}\left<\nu,\mfk\right>\rd\mcH^{n-1}
			=&\sum_{i=1}^L\int_{\Gamma_i}\left\{\left<x,\mu_i\right>\left<\nu,\mfk\right>-\left<x,\nu\right>\left<\mu_i,\mfk\right>\right\}\rd\mcH^{n-2}\notag\\
			=&\sum_{i=1}^L\int_{\Gamma_i}\left\{\left<x,\mu_i\right>\left<\nu,\mfk\right>+\left<x,\nu\right>\frac{-\cos\theta_i}{\sin\theta_i}\left(1+\left<\nu,\mfk\right>\right)\right\}\rd\mcH^{n-2}\notag\\
			=&-\sum_{i=1}^L\int_{\Gamma_i}\left<x,\mu_i\right>\rd\mcH^{n-2},
		\end{align}
		where in the last equality, we have used \eqref{x,n_i} and \eqref{eq-Minko-W-1}.
		
		Combining \eqref{AS16-5.1-Wedge} with \eqref{AS16-5.2-Wedge}, we see that 	
		\begin{align*}
			\int_M\left\{(n-1)-H\left<x,\nu\right>+(n-1)\left<\nu,\mfk\right>\right\}\rd\mcH^{n-1}=0.
		\end{align*}
		This completes the proof.
	\end{proof}
		
	Here we record the following Balancing formula, which was derived when $M$ is $C^2$ (see for example \cite[Lemma 1]{CK16} or \cite[Lemma 7]{LX17}). The proof can be adapted to our singular case by using the approximation argument as in the proof of \cref{PropAS16-2.4}.
	Here we state it without proof.
	
	\begin{lemma}[Balancing formula]
		Let $E$ be as in \cref{PropMinkowski-Poly},
  then for each $i=1,\ldots,L$, it holds that
		\begin{align}\label{balance}
			H\mcH^{n-1}(B^+_i)=\sin\theta_i\mcH^{n-2}(\Gamma_i).
		\end{align}
	\end{lemma}
	
We record the following point-wise formulas in \cite{Souam21}, which will be needed in the proof of \cref{Thm-wedgeShaped}. 
	\begin{lemma}[{\cite[Proof of Theorem 3.4]{Souam21}}]\label{Lem6}
		Assume $H$ is a constant on ${\rm reg}M$ and  the boundary contact angle condition \eqref{condi-Young-W1}holds. Define \begin{align}\label{zeta-wedge}\zeta:=(n-1)-H\left<x,\nu\right>+(n-1)\left<\mfk,\nu\right>\end{align} and \begin{align}\label{Phi-wedge}\Phi:=\frac{H}{n-1}\vert x\vert^2-2\left<x,\nu\right>\end{align} on ${\rm reg}M$. Then
		\begin{enumerate}

			
			\item 
			On ${\rm reg}M$,
			\begin{align}
    &  \De_M\zeta=-H^2-\vert\vert h\vert\vert^2\left(\zeta-(n-1)\right),\label{DeMzeta+BMzeta}
			\end{align}
  \begin{align}\label{DeMPhi}
      \De_M \Phi=2\left(\vert\vert h\vert\vert^2-\frac{H^2}{n-1}\right)\left<x,\nu\right>.
  \end{align}
 \item For each $i$, on ${\rm reg}\Gamma_i$,
			\begin{align}\label{zeta-boudary}
				\left<\nabla^M\zeta,\mu_i\right>-q_i\zeta=0,
			\end{align}
  \begin{align}\label{Phi-boundary}
      \left<\nabla^M\Phi,\mu_i\right>
      =-2\left(\frac{n-2}{n-1}H-\sin\theta_iH_{\Gamma_i}\right)\left<x,\mu_i\right>,
  \end{align}
  where $H_{\Gamma_i}$ is the mean curvature of ${\rm reg}\Gamma_i\subset P_i$.
		\end{enumerate}
	\end{lemma}
 
	\
 
	\begin{proof}[Proof of \cref{Thm-wedgeShaped}]
	Our starting point is \cref{Prop-capillaryCMC-W} and \cref{Prop-StableSet-W}, and we note that these propositions are satisfied by the local minimizers of the free energy functional under volume constraint, see \cref{rmk-local-min}.
	Moreover, we learn from \cref{regularity-thm} that $\vert\vert h\vert\vert\in L^2(M)$, which is crucial in the following proof.
	
	To proceed, notice that the Balancing formula \eqref{balance} shows that the mean curvature of $M$ can not be 0,
	thus we may assume $H>0$ in the following.
		
We shall use $\zeta$ defined by \eqref{zeta-wedge}.	By virtue of the Minkowski-type formula \cref{PropMinkowski-Poly}, we know $\zeta\in C^2({\rm reg}M)$ and $\int_{M}\zeta(x)\rd\mcH^{n-1}(x)=0$. Hence $\zeta$ is a possible candidate for testing the Poincar\'e-type inequality \eqref{result2-wedge}.

	To appeal to \eqref{result2-wedge}, we need to verify the integrability conditions \eqref{condi-integrability-W} for $\zeta$.
First, since $M$ is compact and $H$ is a constant on ${\rm reg}M$, $\zeta\in L^\infty$.
Second,
 \eqref{DeMzeta+BMzeta} and the fact $\vert\vert h\vert\vert\in L^2(M)$ tells that 
$\zeta\De_M\zeta+\vert\vert h\vert\vert^2\zeta^2\in L^2(M)$, while
\eqref{zeta-boudary} tells 
	    $\zeta\left<\nabla^M\zeta,\mu_i\right>-q_i\zeta^2\in L^1(\Gamma)$.
Hence, we may test \eqref{result2-wedge} with $\zeta$ and find that	\begin{align}\label{eqcmain}
			-\int_M\left\{(n-1)\vert\vert h\vert\vert^2-H^2\right\}\left((n-1)-H\left<x,\nu\right>+(n-1)\left<\mfk,\nu\right>\right) \rd\mcH^{n-1}\geq0.
		\end{align}
		Next, we will prove that 
		\begin{align}\label{eqc0}
			\int_{M}\left\{(n-1)\vert\vert h\vert\vert^2-H^2\right\}H\left<x,\nu\right>\rd\mcH^{n-1}=0.
		\end{align}
		To this end, we exploit the function $\Phi$ defined by \eqref{Phi-wedge} on ${\rm reg}M$.
		
		For any $\ep>0$, we have $\varphi_\ep, S^{'}_\ep, S_\ep$ from \cref{Lem-Cut-off}. Let $\tilde{Y}_\ep:\mbR^{n}\ra\mbR^{n}$ be a smooth vector field satisfying 
$\tilde{Y}_\ep=\nabla^M\Phi\text{ on }M\setminus S^{'}_{\ep}. 
		$
		Then let $Y_\ep\in C^1(\mbR^{n};\mbR^{n})$ be the vector field given by
		$
			Y_\ep=\varphi_\ep \tilde{Y}_\ep.$
		Integrating ${\rm div}_M(Y_\epsilon)$ on $M\setminus S'_\ep$ and using the classical divergence theorem,
		we find
		\begin{align}\label{eq-ThmW-1}
		    \int_M\left<\nabla^M\varphi_\ep,\nabla^M\Phi\right>\rd\mcH^{n-1}+\int_{M}\varphi_\ep\De_M\Phi \rd\mcH^{n-1}
      =\sum_{i=1}^L\int_{\Gamma_i}\varphi_\ep\left<\nabla^M\Phi,\mu_i\right>\rd\mcH^{n-2}.
		\end{align}
		
		Now we deal with the term $\int_{\Gamma_i}\varphi_\ep\left<\nabla^M\Phi,\mu_i\right>\rd\mcH^{n-2}$, which is explicitly expressed in \eqref{Phi-boundary}.
  Notice that
		\begin{align*}
		    \sum_{i=1}^L\int_{\Gamma_i}\varphi_\ep\sin\theta_i H_{\Gamma_i}\left<x,\mu_i\right>\rd\mcH^{n-2}
		    =\sum_{i=1}^L\sin\theta_i\cos\theta_i\int_{\Gamma_i}\varphi_\ep H_{\Gamma_i}\left<x,\bar\nu_i\right>\rd\mcH^{n-2},
		\end{align*}
		since the origin $O\in\mfW$ and thanks to the constant contact angle condition \eqref{condi-Young-W}.
		
		On the other hand,
		we consider the position vector field $X(x)=x$, integrating ${\rm div}_{\Gamma_i}(\varphi_\ep X)$ on $\Gamma_i\setminus S'_\ep$ and using the classical divergence theorem, we get
		\begin{align*}
		    \int_{\Gamma_i}(n-2)\varphi_\ep\rd\mcH^{n-2}(x)+\int_{\Gamma_i}\left<\nabla^{\Gamma_i}\varphi_\ep,x\right>\rd\mcH^{n-2}
		    =\int_{\Gamma_i}\varphi_\ep H_{\Gamma_i}\left<x,\bar\nu_i\right>\rd\mcH^{n-2},
		\end{align*}
		which in turn gives that
		\begin{align}\label{eq-ThmW-2}
		    &\sum_{i=1}^L\int_{\Gamma_i}\varphi_\ep\sin\theta_i H_{\Gamma_i}\left<x,\mu_i\right>\rd\mcH^{n-2}\notag\\
		    =&\sum_{i=1}^L\sin\theta_i\cos\theta_i\left\{\int_{\Gamma_i}(n-2)\varphi_\ep\rd\mcH^{n-2}(x)+\int_{\Gamma_i}\left\{\left<\nabla^{\Gamma_i}\varphi_\ep,x\right>\right\}\rd\mcH^{n-2}\right\}.
		\end{align}
		Now we expand \eqref{eq-ThmW-1} by virtue of \eqref{DeMPhi}, \eqref{Phi-boundary} and \eqref{eq-ThmW-2}, then we may send $\ep\searrow0$ and use \eqref{pw-conv},\eqref{esti-cutoff-1},\eqref{esti-cutoff-3},
		the fact that $\vert\vert h\vert\vert\in L^2(M)$, to appeal to the dominated convergence theorem and conclude that
			\begin{align}\label{eq-ThmW-3}
			&\int_M\left(\vert\vert h\vert\vert^2-\frac{H^2}{n-1}\right)\left<x,\nu\right>\rd\mcH^{n-1}\notag\\
			=&-(n-2)\sum_{i=1}^L\left(\cos\theta_i\int_{\Gamma_i}\frac{H}{n-1}\left<x,\onu_i\right>\rd\mcH^{n-2}-\sin\theta_i\cos\theta_i \mcH^{n-2}(\Gamma_i)\right).
		\end{align}
		Since $\mfk$ is a constant vector field, and notice that ${\rm div}_{P_i}\mfk^T=0$, ${\rm div}_M\mfk^T=-H\left<\nu,\mfk\right>$, we may use \eqref{formu-div-B+}, \eqref{formu-div-M} and \eqref{condi-Young-W} to find
		\begin{align}\label{eq-ThmW-4}
		0&=\int_{B_i^+}{\rm div}_{P_i}\mfk^T\rd\mcH^{n-1}
		=\int_{\Gamma_i}\left<\mfk,\bar\nu_i\right>\rd\mcH^{n-1}\notag,
		\\
		    \int_{M}-H\left<\nu,\mfk\right>\rd\mcH^{n-1}
		    &=\sum_i^{L}\int_{\Gamma_i}\left<\mfk,\sin\theta_i\bar N_i+\cos\theta_i\bar\nu_i\right>\rd\mcH^{n-2}
		    =\sum_{i=1}^L\sin\theta_i\cos\theta_i\mcH^{n-2}(\Gamma_i).
		\end{align}
		On the other hand, \eqref{AS16-5.2-Wedge} and the contact angle condition give
		\begin{align}\label{eq-ThmW-5}
		   \int_{M}-H\left<\nu,\mfk\right>\rd\mcH^{n-1}
			=\cos\theta_i\sum_{i=1}^L\int_{\Gamma_i}\frac{H}{n-1}\left<x,\bar\nu_i\right>\rd\mcH^{n-2}.
		\end{align}
		Combining \eqref{eq-ThmW-4} with \eqref{eq-ThmW-5}, we see that the RHS of \eqref{eq-ThmW-3} vanishes, which in turn shows \eqref{eqc0}.

			\eqref{eqcmain} thus becomes
			\begin{align*}
				-\int_M\left\{(n-1)\vert\vert h\vert\vert^2-H^2\right\}\left(1+\left<\mfk,\nu\right>\right) \rd\mcH^{n-1}\geq0.
			\end{align*}
			
   On the other hand, since $\vert\mfk\vert\leq1$,
   we have $1+\left<\mfk,\nu\right>\geq0$, and the Cauchy-Schwarz inequality implies $(n-1)\vert\vert h\vert\vert^2\geq H^2.$ Hence 
			\begin{align*}
				\int_{M}\left\{(n-1)\vert\vert h\vert\vert^2-H^2\right\}\left(1+\left<\mfk,\nu\right>\right)\rd\mcH^{n-1}\geq0.
			\end{align*}
			Therefore we conclude that $(n-1)\vert\vert h\vert\vert^2=H^2$ on ${\rm reg}M$.
			This means, since the equality case of Cauchy-Schwarz inequality happens, that the principal curvatures of $M$ coincide at every point of ${\rm reg}M$, and hence
			it must be locally spherical. 
			
		For $n=3$, ${\rm sing}M=\emptyset$, hence it suffices to consider the possibility of non-intersecting spherical caps, which we shall discuss later.
		Let us consider the situation when $n>3$.
		Since we assume $\mathcal{H}^{n-3}({\rm sing}M)=0$, we exclude the possibility of intersecting but not tangential  spherical caps, which has non-vanishing $\mathcal{H}^{n-3}$ singularities.
			
		Next we exclude the possibility of non-intersecting or tangential spherical caps (or spheres). We only need to handle two such caps (or spheres).
			Let $C_1$ and $C_2$ be two caps and $M=C_1\cup C_2$. Consider on $C_i, i=1,2$ the Robin eigenvalue problem:
			$$-(\Delta_M+\|h\|^2)f=\lambda_1 f \hbox{ on }C_i, \quad \<\nabla^M f,\mu\>=qf \hbox{ on }\p C_i.$$
			It is known that the first Robin eigenvalue $\lambda_1<0$, see \cite[Appendix A]{GWX22}.
			Let us now consider a smooth function on ${\rm reg}M$, given by 
			\begin{equation*}
			    \zeta=\begin{cases}
			       f, &\hbox{ on }C_1,\\
			       -f, &\hbox{ on }C_2.
			    \end{cases}
			\end{equation*}
			It is clear that $\int_{M}\zeta=\int_{C_1}f+\int_{C_2}(-f)=0$. On the other hand, because $\lambda_1<0$, we see that 
			\begin{align*}\label{Poincare-Wedge}
		-\int_M\left(\De_M\zeta+\vert\vert h\vert\vert^2\zeta\right)\zeta \rd\mcH^{n-1}+\int_{\Gamma}\left(\left<\nabla^M\zeta,\mu\right>-q\zeta\right)\zeta \rd\mcH^{n-2}<0,
	\end{align*}
	which is a contradiction to the Poincar\'e-type inequality \eqref{result2-wedge}.
	
	For the case $M=C_1\cup S$ for some spherical cap $C_1$ and a closed sphere $S$, we may use a similar contradiction argument by choosing \begin{equation*}
			    \zeta=\begin{cases}
			       & f, \hbox{ on }C_1,\\
			        &c, \hbox{ on }S,
			    \end{cases}
	\end{equation*}
	where  $c$ is a constant such that $\int_M\zeta=0$, to conclude that this is not possible.
	Thus we have proved that $M$ is a spherical cap.
			
	Finally, we finish our proof by showing that as a special case when $\mfW$ is just a half-space, the condition $\vert\mfk\vert=|\cos\theta|\leq1$ is trivially valid,
		and hence the proof above holds in this case.
		This completes the proof.
		\end{proof}
		
		\begin{remark}
			\normalfont
			When each $\theta_i= \frac{\pi}{2}$, it is apparent that $\vert\mfk\vert=0$, which means that \cref{Thm-wedgeShaped} holds true in this situation.
			In particular, this generalizes the results for the smooth stable free boundary capillary hypersurface in a wedge of L\'opez \cite[Theorem 1]{Lopez14} to the non-smooth case.
		\end{remark}
		
	\section{Volume-constraint local minimizers in a  ball}\label{Sec-6}
		In this section, we prove \cref{umbillic} for the case $\O=\mfB^n$.
  The proof follows largely from \cite{WX19}.
		
		Let us first recall that a key ingredient in \cite{WX19} is the following conformal Killing vector field: fix any $a\in\mbR^n$,
		we consider a smooth vector field $X_a$ in $\rr^n$ defined by
		\begin{align*}
			X_a(x)=\left<x, a\right>x-\frac{1}{2}(\vert x\vert^2+1)a, \quad x\in\rr^n.
		\end{align*}
		Notice that $X_a$ is a conformal Killing vector field in $\rr^n$ which is tangent to $\ss^{n-1}$, in fact, this is proved in the following:
		\begin{lemma}[{\cite[Proposition 3.1]{WX19} }]\label{Lemma5.1}
			On $\mfB^n$, there holds:
			\begin{enumerate}
				\item $\mcL_{X_a}g_{\rm euc}=\left<x, a\right>g_{\rm euc}$, or equivalently
				\begin{align*}
					\frac{1}{2}[\nabla_i(X_a)_j+\nabla_j(X_a)_i]=\left<x, a\right>g_{ij}.
				\end{align*} 
				\item $\left<X_a,x\right>\mid_{\p\mbB^n}=0$.
			\end{enumerate}
			Here $\mcL$ denotes the Lie derivative, $g_{\rm euc}$ denotes the canonical Riemannian metric in $\mbR^n$, $\nabla_i(X)_j:=g_{\rm euc}(\nabla_{e_i}X,e_j)$, where $\{e_i\}_{i=1,\ldots,n}$ denote the coordinate vectors of $\mbR^n$.
		\end{lemma}
		Similar with \cref{PropMinkowski-Poly}, we can extend the classical Minkowski-type formula for smooth capillary hypersurfaces in $\mfB^n$ \cite[Proposition 3.2]{WX19} to our singular case,
		by adapting the approximation argument in \cref{PropMinkowski-Poly} that is based on the cut-off functions constructed in \cref{Lem-Cut-off}, and exploiting \cref{PropAS16-2.4}.
		Here we state it without proof.
		\begin{proposition}[Minkowski-type formula in a ball]\label{Prop-Minko-Ball}
  		Let $E\subset\mfB^n$ be a set of finite perimeter, which is stationary for $\mcF_\beta$ under volume constraint. Assume that $M$ and $\Gamma$ satisfy the local Euclidean volume growth condition \eqref{ineq-vg-M} and \eqref{ineq-vg-Gamma} respectively. Assume in addition that  $\mcH^{n-2}({\rm sing}M)=0$
    and $H\in L^1(M)$,
        then there holds
		\begin{align}\label{MinkowskitypeIneq}
		\int_{M}(n-1)\left<x+\cos\theta \nu,a\right>-H\left<X_a,\nu\right>\rd\mcH^{n-1}=0.
		\end{align}
		\end{proposition}
  
		We record some point-wise computations in \cite{WX19}, which are valid on the regular part of $M$.

	\begin{lemma}[{\cite[Proposition 3.5]{WX19}}]
		Assume that $H$ is a constant on ${\rm reg}M$ and the boundary contact angle condition \eqref{condi-Young-W1} holds. Define \begin{align}\label{varphi-a}	\varphi_a:=(n-1)\left<x+\cos\theta \nu,a\right>-H\left<X_a,\nu\right>\end{align} and \begin{align}\label{Phi-ball}\Phi:=\frac{1}{2}(\vert x\vert^2-1)H-(n-1)\left(\left<x,\nu\right>+\cos\theta\right)\end{align} on ${\rm reg}M$. Then
		\begin{enumerate}

			
			\item 
			On ${\rm reg}M$,
			\begin{align}
    &  \De_M\varphi_a+\|h\|^2\varphi_a=[(n-1)\|h\|^2-H^2]\<x, a\>,\label{DeMzeta+BMzeta-ball}
			\end{align}
  \begin{align}\label{DeMPhi-ball}
      \Delta_M\Phi=\left((n-1)\vert\vert h\vert\vert^2-H^2\right)\left<x,\nu\right>.
  \end{align}
 \item On ${\rm reg}\Gamma$,
			\begin{align}\label{zeta-boudary-ball}
				\left<\nabla^M\varphi_a,\mu\right>-q\varphi_a=0,\quad \Phi=0,
			\end{align}
   where \begin{align}\label{q1}
	q=\frac{1}{\sin\theta}+\cot\theta h(\mu,\mu ).
		\end{align}
		\end{enumerate}
	\end{lemma}
 
		\
		
		\begin{proof}[Proof of \cref{umbillic} for $\O=\mfB^n$]
		Our starting point is \cref{Prop-capillaryCMC} and \cref{Prop-StableSet}, and we note that these propositions are satisfied by the local minimizers of the free energy functional under volume constraint, see \cref{rmk-local-min}.
	Moreover, we learn from \cref{regularity-thm} that $\vert\vert h\vert\vert\in L^2(M)$, which is crucial in the following proof.

 We shall use $\varphi_a:{\rm reg}M\ra\mfR^1$ defined by \eqref{varphi-a}.
			By virtue of the Minkowski type formula \eqref{MinkowskitypeIneq},  we have
			\begin{align}\label{integralofvarphia}
				\int_M\varphi_a\rd\mcH^{n-1}=0.
			\end{align}
			The integrability conditions \eqref{condi-integrability} for $\varphi_a$ can be verified similarly as in the proof of \cref{Thm-wedgeShaped}.
			Using this test function to test the Poincar\'e-type inequality \eqref{ineq-Poincare}, by virtue of \eqref{DeMzeta+BMzeta-ball} and \eqref{zeta-boudary-ball},  we arrive at
			\begin{align}\label{WX324}
				\int_M \left((n-1)\vert x\vert^2+(n-1)\cos\theta\left<x,\nu\right>-\frac{1}{2}(\vert x\vert^2-1)H\left<x,\nu\right>\right)
				\cdot\left[(n-1)\vert\vert h\vert\vert^2-H^2\right] \rd\mcH^{n-1}\leq0.
			\end{align}
			
			
			To proceed, we consider the function $\Phi$ defined in \eqref{Phi-ball}.
			We follow closely the approximating argument as the one in the proof of \cref{Thm-wedgeShaped}, integrating ${\rm div}_M\left(\varphi_\ep\nabla^M\left(\frac{1}{2}\Phi^2\right)\right)$ on $M\setminus S'_\ep$ and using the classical divergence theorem. By virtue of \eqref{pw-conv},\eqref{esti-cutoff-1}, the fact that $\vert\vert h\vert\vert\in L^2(M)$, and the fact that $\Phi=0$ on ${\rm reg}\Gamma$, we may send $\ep\searrow0$ and use the dominated convergence theorem to get	\begin{align}\label{WX19-3.27}
				\int_M\De_M\left(\frac{1}{2}\Phi^2\right)\rd\mcH^{n-2}=\int_{\Gamma}\Phi\nabla_\mu\Phi\rd\mcH^{n-2}
				=0.
			\end{align}
			Adding \eqref{WX19-3.27} to \eqref{WX324},
			expanding $\De_M\left(\frac{1}{2}\Phi^2\right)$ and using \eqref{DeMPhi-ball},
			we find
			\begin{align}\label{WXfinal}
				0\geq	
				\int_M(n-1)\vert x^T\vert^2\left((n-1)\vert\vert h\vert\vert^2-H^2\right)+\vert\nabla^M\Phi\vert^2\rd\mcH^{n-1}\geq0,
			\end{align}
			where $x^T$ is the tangential part of $x$ with respect to ${\rm reg}M$ and in the last inequality we have used the fact that $(n-1)\vert\vert h\vert\vert^2-H^2$ is non-negative by virtue of the Cauchy Schwarz inequality. 
			
			From \eqref{WXfinal} we proceed exactly as in the proof of \cite[Theorem 3.1]{WX19} to conclude that
			\begin{align*}
				(n-1)\vert\vert h\vert\vert^2=H^2\quad, \forall x\in {\rm reg}M.
			\end{align*}
			By virtue of the Cauchy-Schwarz inequality again, the equality holds if and only if ${\rm reg}M$ is umbilical in $\mbB^n$.
   
   If the constant mean curvature $H\neq0$, we know that ${\rm reg}M$ is spherical. By applying the same argument as that in the proof of \cref{Thm-wedgeShaped}, we conclude that $M$ must be a spherical cap.
If $H=0$,  we know that ${\rm reg}M$ is flat.
			Similarly, we exclude the possibility of intersecting $(n-1)$-balls by virtue of the fact that $\mcH^{n-3}({\rm sing}M)=0$.
		To exclude the possibility of non-intersecting $(n-1)$-balls, again we use the Robin eigenvalue problem on each $(n-1)$-ball as in the proof of \cref{Thm-wedgeShaped} and construct a function $\zeta$ on $M$ which violates the stability of $M$ if $M$ is not connected.
		In particular, this shows that $M$ must be a single plane, and thus completes the proof.
			
			\end{proof}
		\printbibliography
		
	\end{document}